\documentclass[12pt]{amsart}
\usepackage[latin1]{inputenc}
\usepackage{mathtools}
\usepackage{amsxtra}
\usepackage{amsmath}
\usepackage{amscd}
\usepackage{amssymb}
\usepackage{amsfonts}
\usepackage[all]{xy}
\usepackage{mathrsfs}
\usepackage{amsthm}
\usepackage{color}
\usepackage{upgreek}
\usepackage{verbatim}
\usepackage{enumerate}
\usepackage{latexsym}
\usepackage{graphicx}
\usepackage{todonotes}
\usepackage{setspace}
\usepackage{stmaryrd}
\usepackage{nicefrac}
\usepackage{euscript}
\usepackage{mathdots}
\usetikzlibrary{decorations.pathmorphing}
\usepackage
{hyperref}

 \definecolor{darkgreen}{HTML}{336633}
 \definecolor{darkred}{HTML}{993333}

\newcommand{\arxiv}[1]{\href{http://arxiv.org/abs/#1}{\tt
    arXiv:\nolinkurl{#1}}}

\hypersetup{colorlinks,linkcolor=blue,urlcolor=cyan,citecolor=blue}

\theoremstyle{definition}
\newtheorem{thm}{Theorem}[section]
\newtheorem{lem}[thm]{Lemma}
\newtheorem{prop}[thm]{Proposition}

\newtheorem{cor}[thm]{Corollary}
\numberwithin{equation}{section}

\theoremstyle{definition}
\newtheorem{df}[thm]{Definition}

\theoremstyle{remark}
\newtheorem{rem}[thm]{Remark}

\def\bbC{\mathbb{C}}
\def\bbJ{\mathbb{J}}

\def\bbZ{\mathbb{Z}}

\def\scrD{\mathscr{D}}
\def\scrN{\mathscr{N}}
\def\scrR{\mathscr{R}}

\def\scrZ{\mathscr{Z}}

\def\calC{\mathcal{C}}

\def\calP{\mathcal{P}}
\def\calU{\mathcal{U}}
\def\calV{\mathcal{V}}

\def\TPL{\mathcal{TP}}

\def\GL{GL}

\def\ALB{\mathrm{LB}^-}

\def\la{\lambda}
\def\ov{\overline}
\def\diag{\mathrm{diag}}

\newcommand{\rank}{\operatorname{rank}\nolimits}

\title[Anti-commuting varieties]{Anti-commuting varieties}
\author[Xinhong Chen]{Xinhong Chen}
\address{ Department of Mathematics\\ Southwest Jiaotong University\\Sichuan 611756, China}
\email{chenxinhong@swjtu.edu.cn}
\author[Weiqiang Wang]{Weiqiang Wang}
\address{ Department of Mathematics\\ University of Virginia\\ Charlottesville, VA 22904, USA}
\email{ww9c@virginia.edu}


\begin{document}

\begin{abstract}
We study the anti-commuting variety which consists of pairs of anti-commuting $n\times n$ matrices. We provide an explicit description of its irreducible components and their dimensions. The GIT quotient of the anti-commuting variety with respect to the conjugation action of $GL_n$ is shown to be of pure dimension $n$. We also show the semi-nilpotent anti-commuting variety (in which one matrix is required to be nilpotent) is of pure dimension $n^2$ and describe its irreducible components.
\end{abstract}

\maketitle


 \setcounter{tocdepth}{1}
 \tableofcontents

 \section{Introduction}

 \subsection{The commuting varieties}

 The variety of pairs of commuting $n\times n$ matrices, known as the commuting variety, was shown by Motzkin and Taussky \cite{MT55} and Gerstenhaber \cite{Ger61} to be irreducible.  The notion of commuting varieties has been extended since then to semisimple Lie algebras and variants of nilpotent commuting varieties have been studied, in characteristic zero as well as in positive characteristic; for samples see \cite{Ri79, Ba01, Pr03} and references therein.

 \subsection{The goal}

In this paper we introduce the anti-commuting variety which consists of pairs of $n\times n$ matrices over $\bbC$ by
\begin{equation*}
\scrZ_n =\{ (A, B) \in M_n(\bbC) \times M_n(\bbC) \mid AB+BA=0 \}.
\end{equation*}
The general linear group $GL_n$ acts on $\scrZ_n$ by diagonal conjugation.
Our main goal is to understand the geometric structures of $\scrZ_n$, the GIT quotient $\scrZ_n// GL_n$, and a so-called semi-nilpotent anti-commuting variety $\scrN_n$; see \eqref{eq:Nn} below.


 %
 %
 \subsection{The main results}

In contrast to the commuting variety case, the variety $\scrZ_n$ turns out to be reducible.
A closed subvariety $\scrZ_{p,m,r}$ of $\scrZ_n$ is defined in Definition~\ref{def:Zpmr},
where the triples $(p,m,r)$ of nonnegative integers in the indexing set $\TPL_n$ given in \eqref{TPL} satisfy $2p+m+r=n$. Our first main result is the following structure theorem for $\scrZ_n$.

\begin{thm}  [Theorem~\ref{thm:AC}]
 \label{thm:1Zn}
The variety $\scrZ_n$ has one irreducible component $\scrZ_{p,m,r}$, for each $(p,m,r) \in \TPL_n$, so that $\scrZ_n =\cup_{(p,m,r) \in \TPL_n} \scrZ_{p,m,r}$.
Moreover, we have $\dim \scrZ_{p,m,r} =n^2+p$.
\end{thm}
A simple count of \eqref{TPL} shows that the number of the irreducible components in $\scrZ_n$ is $(k+1)^2$ if $n=2k$, and $k(k+1)$ if $n=2k-1$.

We also study the GIT quotient $\scrZ_n// GL_n$.

\begin{thm}  [Theorem~\ref{thm:pure dimension}]
  \label{thm:2GIT}
The variety $\scrZ_n//GL_n$ is of pure dimension $n$; that is, every irreducible component  $\scrZ_{p,m,r}//GL_n$ has the same dimension $n$, for $(p,m,r) \in \TPL_n$.
\end{thm}

Introduce the following semi-nilpotent subvariety of $\scrZ_n$:
\begin{equation}
  \label{eq:Nn}
\scrN_n =\{ (A, B) \in M_n(\bbC) \times M_n(\bbC) \mid AB+BA=0, \text{ for } A \text{ nilpotent} \}.
\end{equation}
For a partition $\la$ of $n$, let $J_\la(0)$ be the associated Jordan normal form of eigenvalue $0$. We denote by $\scrN_\la$ the Zariski closure of the union of $G$-orbits of $(J_\la(0),B)$ in $\scrN_n$.

\begin{thm}
[Lemma~\ref{lem: dim of Nlambda}, Theorem~\ref{thm:N}]
\label{thm:3nil}
The variety $\scrN_n$ has irreducible components $\scrN_\la$ for all partitions of $n$.
Moreover, $\scrN_\la$ for all $\la$ have the same dimension $n^2$.
\end{thm}
For quivers with loops, a certain variety was first studied by Yiqiang Li \cite{Li16} and then in \cite{Boz16}, and it provides a suitable generalization of Lusztig's nilpotent variety for quivers without loops \cite{Lu91}. For the Jordan quiver, Li's variety reduces to  the semi-nilpotent subvariety in the commuting variety setting, and Li proved  the counterpart of Theorem~\ref{thm:N}. It is interesting to note that the moment map used by Lusztig for quivers without loops can be formulated via either commutators or anti-commutators. From this view, the anti-commutator considered in this paper is also a natural extension to the Jordan quiver (and it might be interesting to explore this further for more general quivers with loops).

 \subsection{Further questions}
   \label{subsec:questions}

  Let us make a remark on the ground field, which we have assumed to be $\bbC$. Results in Sections~\ref{sec:Jordan}, \ref{sec:Irred}, \ref{sec:cover} and \ref{sec:nilpotent} are valid over an algebraically closed field $K$ of characteristic not equal to 2. We expect that Theorem~\ref{thm:1Zn} is valid over $K$ too. While our current proof uses the GIT quotients to show that the varieties $\scrZ_{p,r,m}$ have no inclusion relation, likely a different argument valid over $K$ which bypasses GIT can be found.

It is enlightening to view the commuting variety (and respectively, anti-commuting variety) as the variety of modules of dimension $n$  of the algebra $\calP_2$ of polynomials (and respectively, algebra $\calP_2^-$ of skew-polynomials) in 2 variables.  Both these algebras can be regarded as the specializations at $q=\pm 1$ of the $q$-Weyl algebra
\[
xy =qyx.
\]
This leads to the natural question of understanding the geometric structures of the variety of modules of dimension $n$ of the $q$-Weyl algebra, with $q$ being an arbitrary root of 1. This will be developed in a sequel \cite{CL19}.

It will be interesting to understand the stratification of the varieties $\scrZ_{p,m,r}$.

The anti-commuting and commuting varieties can also be related directly in different ways. See \S\ref{subsec:C-AC} for some suggestions.

 \subsection{The organization}

The paper is organized as follows.
We describe in Section~\ref{sec:Jordan} the matrices which anti-commute (and respectively, commute) with a Jordan normal form. The similarities and differences between anti-commuting and commuting cases are made clear.

In Section~\ref{sec:Irred}, we prove that the variety $\scrZ_{p,m,r}$ is irreducible and compute its dimension. Two general properties for the varieties $\scrZ_{p,m,r}$ are useful. A ``direct sum" property  (which is proved in Lemma~\ref{lem:ZZZ12}) states that $\big(\diag(A_1,A_2), \diag(B_1,B_2) \big) \in\scrZ_{p_1+p_2,m_1+m_2,r_1+r_2}$,  for $(A_1,B_1)\in\scrZ_{p_1,m_1,r_1}$ and $(A_2,B_2)\in\scrZ_{p_2,m_2,r_2}$. A flipping symmetry $\scrZ_{p,m,r} \cong \scrZ_{p,r,m}$ by sending $(A,B) \mapsto (B,A)$ is established in Lemma~\ref{lem:switch}.

Section~\ref{sec:cover} is devoted to the proof that $\scrZ_n$ is the union of $\scrZ_{p,m,r}$ for all $(p,m,r)\in \TPL_n$. By a reduction via the ``direct sum" property, we only need to figure out which component $\scrZ_{p,m,r}$ contains a given $(A,B) \in \scrZ_n$, where either $A$ has eigenvalues $\pm a$ for a fixed $a\in \bbC^*$ or $A$ is nilpotent (see Propositions~\ref{prop:block a,-a} and \ref{prop:Jordan block 0}).
These two cases require different approaches, using only elementary algebraic geometry and linear algebra arguments. There are often natural (and simpler) commutative counterparts to our statements in our anti-commuting setting, some of which could be new and applicable in other circumstances; one such application is indeed found in \cite{WW18}.

It is established in Section~\ref{sec:GIT} that the GIT quotient $\scrZ_n//GL_n$ is of pure dimension $n$ (Theorem~\ref{thm:2GIT}). This is in turn used to show that there is no inclusion relation among $\scrZ_{p,m,r}$ for different triples $(p,m,r)$, completing the proof that $\scrZ_{p,m,r}$ are irreducible components of $\scrZ_n$, whence Theorem~\ref{thm:1Zn}.

Theorem~\ref{thm:3nil} on the semi-nilpotent anti-commuting variety is proved in Section~\ref{sec:nilpotent}.

 \vspace{.4cm}
 {\bf Acknowledgment.}
We thank Ming Lu for numerous stimulating discussions, and Zhenbo Qin for helpful remarks. We thank Yiqiang Li for his suggestion of considering semi-nilpotent varieties and bringing his work to our attention. We thank the referee for constructive comments and insightful suggestions, which greatly improve this paper. The first author is partially supported by NSFC grant No.~11601441 and CSC grant No.~201707005033, and she thanks University of Virginia for hospitality during her visit when this project is carried out. The second author is partially supported by an NSF grant DMS-1702254.

%

\section{Matrices (anti-)commuting with Jordan normal forms}
  \label{sec:Jordan}

In this section, we describe uniformly the matrices which anti-commute or commute with a matrix in Jordan normal form.

\subsection{$\sigma$-layered matrices}

Let $\sigma\in \{\pm 1\}$; sometimes we identify $\sigma$ as a sign $\pm$. Consider $m\times n$ matrices over $\bbC$ of the form:
\begin{eqnarray}
   \label{eq:block a1}
{\rm L}^\sigma(m,n,\vec{v})=\left [\begin{array}{cccccccccc}
 0      & \cdots  & 0      & b_1     & b_2        &\cdots     &  b_{m-1}             & b_m                 \\
 0      & \cdots  & 0      & 0       & \sigma b_1       &\cdots     & \sigma b_{m-2}             &\sigma b_{m-1}           \\
 \vdots &         & \vdots & \vdots  &            & \ddots    & \vdots               & \vdots             \\
 0      & \cdots  & 0      & 0       &  0         &\cdots     & \sigma^{m-2}b_1          &\sigma^{m-2}b_2       \\
 0      & \cdots  & 0      & 0       & 0          &\cdots     & 0                    &\sigma^{m-1}b_1
\end{array} \right] \quad \text{if } m\leq n,
\end{eqnarray}
or
\begin{eqnarray}
   \label{eq:block a2}
{\rm L}^\sigma(m,n,\vec{v})=\left [\begin{array}{cccccccccc}
  b_1     & b_2      &    \cdots       & b_{n-1}           &b_n                \\
   0      &\sigma b_1      & \cdots          &\sigma b_{n-2}          &\sigma b_{n-1}            \\
     \vdots     &   \vdots       &  \ddots         &               \vdots    &        \vdots            \\
   0      &    0     &       \cdots          & \sigma^{n-2} b_1        &    \sigma^{n-2} b_2       \\
   0      &     0    &   \cdots        &     0             &   \sigma^{n-1} b_1    \\
   0      &    0     &   \cdots        &   0               &  0                  \\
 \vdots   &\vdots    &                 &\vdots             &\vdots               \\
  0       &    0     &    \cdots       &    0              &  0
 \end{array} \right] \quad \text{if } m\geq n,
\end{eqnarray}
where $\vec{v}=(b_1,\ldots,b_{\min(m,n)})$.
We call matrices of the form ${\rm L}^\sigma(m,n,\vec{v})$ {\em $\sigma$-layered matrices} (or layered matrices for $\sigma=+$ and anti-layered matrices for $\sigma=-$).

Let $J_m(\alpha)$ denote the $m\times m$ Jordan block of eigenvalue $\alpha$. The following lemma is a special case of Lemmas~\ref{lem:AC2dist} and \ref{lem:AC2a}, where the reader can find proofs.

\begin{lem}
   \label{Jordan anti-layer}
Let $\alpha_1, \alpha_2 \in \bbC$, and let $\sigma\in \{\pm 1\}$.  Then a matrix $B$ satisfies the identity $J_{m}(\alpha_1) B =\sigma B J_{n}(\alpha_2)$ if and only if
\begin{eqnarray*}
B=\begin{cases}
  0                                                       & \   \text{if}\  \alpha_1\neq \sigma\alpha_2 \\
   {\rm L}^\sigma(m,n,\vec{v})                                        & \  \text{if}\   \alpha_1=\sigma\alpha_2,
\end{cases}
\end{eqnarray*}
for some vector $\vec{v}$; see \eqref{eq:block a1}--\eqref{eq:block a2}.
\end{lem}

\begin{prop}\label{ACmatrix}
Let $\sigma\in \{\pm 1\}$. Let $A=\diag(J_{m_1}(\alpha_1),J_{m_2}(\alpha_2),...,J_{m_r}(\alpha_r))$. Then a matrix $B$ satisfies $AB=\sigma BA$ if and only if $B=(B_{ij})$ is an $r\times r$ block matrix with its $m_i\times m_j$-submatrix
\begin{eqnarray*}
B_{ij}=\begin{cases}
0                                                       &   \text{if}\ \alpha_i\neq \sigma\alpha_j \\
{\rm L}^\sigma( m_i,m_j,\vec{v}_{ij})                               &   \text{if}\    \alpha_i=\sigma\alpha_j,
\end{cases}
\end{eqnarray*}
for arbitrary $\min (m_i,m_j)$-tuple vectors $\vec{v}_{ij}$ for all $i,j$; see \eqref{eq:block a1}--\eqref{eq:block a2}.
\end{prop}

\begin{proof}
The $(i,j)$-th block of the equation $AB-\sigma BA=0$ reads $ J_{m_i}(\alpha_i)B_{ij}-\sigma B_{ij}J_{m_j}(\alpha_j)=0$. The proposition now follows by Lemma~\ref{Jordan anti-layer}.
\end{proof}

\subsection{$\sigma$-layered block matrices}
  \label{subsec:block}

Introduce the following $sn\times sn$ matrix in $s$ blocks
\begin{align*}
&\mathbb{J}_{s,n}(0)=
{
\begin{bmatrix}
0      &  I_n  & 0     & \ldots  &  0      & 0              \\
0      &  0    &I_n    & \ldots  &  0      & 0              \\
\vdots &\vdots &\vdots &         & \vdots  &  \vdots   \\
0      &  0    &0      & \ldots  &  0      & I_n             \\
0      &  0    &0      & \ldots  &  0      & 0              \\
 \end{bmatrix},
 }
\end{align*}
where $I_n$ is the $n\times n$ identity matrix. Note $\mathbb{J}_{1,n}(0)$ is the $n\times n$ zero matrix. For any $\alpha \in \bbC$, we let
\[
\mathbb{J}_{s,n}(\alpha) = \alpha I_{sn} +\mathbb{J}_{s,n}(0).
\]
We introduce a shorthand notation
\[
J_s^n(\alpha) :=\diag(\underbrace{J_s(\alpha),\ldots,J_s(\alpha)}_{n}).
\]
Note $J_s^n(\alpha) = \alpha I_{sn} + J_s^n(0)$.

\begin{lem}
  \label{lem:sJordan}
The matrix $J_s^n(\alpha)$ is similar to the  matrix $\mathbb{J}_{s,n}(\alpha)$.
\end{lem}

\begin{proof}
It suffices to show the lemma for $\alpha=0$, which follows by comparing the ranks of $(J_s^n(0))^k$ and $(\mathbb{J}_{s,n}(0))^k$, for all $k\ge 1$.
\end{proof}

Let $\sigma\in \{\pm 1\}$. Consider the following $s\times t$ block matrices
\begin{align}\label{eq:block al1}
{\rm LB}^\sigma(s,t,\vec{V})=
\begin{bmatrix}
B_1      &  B_2     & \ldots  & B_{t-1}    & B_t             \\
0        & \sigma B_1     & \ldots  &\sigma B_{t-2}    & \sigma B_{t-1}           \\
\vdots   &\vdots    &         &\vdots      &  \vdots   \\
0        &  0       & \ldots  &\sigma^{t-2}B_1 &  \sigma^{t-2}B_2          \\
0        &  0       & \ldots  & 0          &  \sigma^{t-1}B_1      \\
0        &  0       & \ldots  &0           &  0      \\
\vdots   & \vdots   &       &\vdots           & \vdots      \\
0&0& \ldots &0&0
 \end{bmatrix}, \text{ if}\ s\geq t;
\end{align}
or
\begin{align}\label{eq:block al2}
{\rm LB}^\sigma (s,t,\vec{V})=
\begin{bmatrix}
0      &\ldots   & 0     & B_1      &  B_2       & \ldots   & B_{s-1}    & B_s             \\
0      & \ldots  & 0     & 0        & \sigma B_1       & \ldots   & \sigma B_{s-2}   & \sigma B_{s-1}           \\
\vdots &                 &  \vdots  & \vdots   & \vdots   &          &\vdots      &  \vdots   \\
0      &\ldots   &0      & 0        & 0          &  \ldots  & \sigma ^{s-2}B_{1}& \sigma ^{s-2} B_{2} \\
0      & \ldots  &0      & 0        & 0          &  \ldots  &  0         & \sigma^{s-1}B_1\\
 \end{bmatrix}, \text{ if}\ s\leq t,
\end{align}
where $\vec{V}=(B_1,\dots,B_{\min(s,t)})$. Note that the $B_{i}$'s are matrices of the same size. We call the block matrices of the form ${\rm LB}^\sigma(s,t,\vec{V})$ \emph{$\sigma$-layered block matrices}  (or layered block matrices for $\sigma=+$ and anti-layered block matrices for $\sigma=-$).

The following lemma and proposition are natural generalizations of Lemma~ \ref{Jordan anti-layer} and Proposition~\ref{ACmatrix}. The proofs are again by direct computations and will also be skipped.

\begin{lem}
 \label{lem:ALB}
 Let $\alpha_1, \alpha_2 \in \bbC$, and let $\sigma\in \{\pm 1\}$.
Then a matrix $B$ satisfies the identity $\mathbb{J}_{s,n_s}(\alpha_1)B=\sigma B\mathbb{J}_{t,n_t}(\alpha_2)$ if and only if
\begin{eqnarray*}
B=\begin{cases}
  0                                                       & \   \text{if}\  \alpha_1\neq \sigma\alpha_2 \\
   {\rm LB}^\sigma(s,t,\vec{V})                                        & \  \text{if}\   \alpha_1=\sigma\alpha_2,
\end{cases}
\end{eqnarray*}
for $\vec{V}=(B_1,\dots,B_{\min(s,t)})$ with $n_s\times n_t$ matrices $B_i$ for all $i$; see \eqref{eq:block al1}-\eqref{eq:block al2}.
\end{lem}

\begin{prop}\label{ALBmatrix}
Let $\sigma\in \{\pm 1\}$. Let $A=\diag(\mathbb{J}_{1,n_1}(\alpha_1), \mathbb{J}_{2,n_2}(\alpha_2),..., \mathbb{J}_{r,n_r}(\alpha_r))$. Then a matrix $B$ satisfies $AB= \sigma BA$ if and only if $B=(B_{ij})$ is an $r\times r$ block matrix with its $in_i\times jn_j$-submatrix
\begin{eqnarray*}
B_{ij}=\begin{cases}
0                                                       &   \text{if}\ \alpha_i\neq \sigma\alpha_j \\
{\rm LB}^\sigma ( i,j,\vec{B}_{ij})                               &   \text{if}\    \alpha_i=\sigma\alpha_j;
\end{cases}
\end{eqnarray*}
here ${\rm LB}^\sigma( i,j,\vec{B}_{ij})$ is an $i\times j$ block matrix (whose block size is $n_i\times n_j$) for every $i$ and $j$; see \eqref{eq:block al1}--\eqref{eq:block al2}.
\end{prop}

\section{Irreducibility and dimension of the variety $\scrZ_{p,m,r}$}
  \label{sec:Irred}

In this section, we shall prove that the variety $\scrZ_{p,m,r}$ introduced in Definition~\ref{def:Zpmr}
is irreducible, and then compute its dimension.

\subsection{Definitions of $\scrZ_n$, $\scrZ_{p,m,r}$ and $\scrR_A$}

Denote by $M_n(\bbC)$ the space of $n\times n$ matrices over $\bbC$. The {\em anti-commuting variety} is defined to be
\begin{equation}
  \label{eq:Zn}
\scrZ_n =\{ (A, B) \in M_n(\bbC) \times M_n(\bbC) \mid AB+BA=0 \}.
\end{equation}
The group
\[
G= GL_n
\]
acts on $\scrZ_n$ by $g. (A,B) =(gAg^{-1}, gBg^{-1})$.

For an $n\times n$ matrix $A$, denote the set of matrices which anti-commute with $A$ by
\begin{equation}
  \label{eq:AC-A}
 \scrR_A= \{B  \in M_n(\bbC) \mid AB+BA=0 \}.
\end{equation}

Introduce the indexing set
\begin{equation} \label{TPL}
\TPL_n :=\{(p,m,r) \in \bbZ_{\geq0}^3\mid 2p+m+r=n \}.
\end{equation}
Recall $\scrZ_n$ denotes the variety of anti-commuting pairs of $n\times n$ matrices.

\begin{df}
  \label{def:Zpmr}
Let $(p,m,r) \in \TPL_n$.
The (constructible) set $\scrZ_{p,m,r}^\circ$ is defined to be the union of $G$-orbits of $(A,B) \in \scrZ_n$ where $A$ is an arbitrary diagonal matrix of the form
\begin{equation}\label{eq:A}
A=\diag(a_1, -a_1, a_2, -a_2, \ldots, a_p, -a_p,  a_{p+1}, \ldots, a_{p+m},\underbrace{0,\ldots, 0}_r),
\end{equation}
with $a_i\neq \pm a_j$ and $a_i\in \bbC^*$ for any $1\leq i\neq j\leq p+m$.

The variety $\scrZ_{p,m,r}$ is defined to be the Zariski closure of $\scrZ_{p,m,r}^\circ$.
\end{df}

\subsection{The irreducibility of $\scrZ_{p,m,r}$}

Recall $G= GL_n$. Define a map
\begin{equation}
  \label{eq:phi}
\phi_{p,m,r}: G\times \bbC^{p+m} \times \bbC^{2p+r} \longrightarrow \scrZ_n
\end{equation}
by letting
\[
\phi_{p,m,r} \big( g, (a_1,\dots,a_{p+m}), (b_1,c_1,\dots,b_p,c_p,b_{p+1},\dots, b_{p+r}) \big) =(gAg^{-1},gBg^{-1}),
\]
where $A$ is given by \eqref{eq:A} and
\[
B= \rm{diag} \Big(\begin{bmatrix}
0& b_1
\\ c_1&0
\end{bmatrix},\dots,
\begin{bmatrix}
 0& b_p\\
  c_p&0
  \end{bmatrix},
\underbrace{0,\dots,0}_m, b_{p+1},\dots,b_{p+r} \Big).
\]
Then $\phi_{p,m,r}$ is a morphism of algebraic varieties.

\begin{prop}
   \label{prop:irred}
The variety $\scrZ_{p,m,r}$ is irreducible, for each $(p,m,r)\in \TPL_n$. Moreover, we have $\ov{{\rm Im}(\phi_{p,m,r})}= \scrZ_{p,m,r}$.
\end{prop}

\begin{proof}
Let $T' \subset G\times \bbC^{p+m} \times \bbC^{2p+r}$ denote the open subset which consists of the triples
\[
\big( g, (a_1,\dots,a_{p+m}), (b_1,c_1,\dots,b_p,c_p,b_{p+1},\dots, b_{p+r}) \big)
\]
such that $a_i\neq \pm a_j$ and $a_i\in \bbC^*$ for $i\neq j$. Clearly $T'$ is irreducible.
By definition, $\scrZ_{p,m,r}$ is the closure of $\phi_{p,m,r} (T')$. Hence, $\scrZ_{p,m,r}$ is irreducible.
\end{proof}
\subsection{Dimension of $\scrZ_{p,m,r}$}

For any $A=(a_{ij})_{n\times n}\in M_n(\bbC)$, denote by
\[
\chi^A(X)=X^n+ \chi^A_{n-1} X^{n-1}+\cdots + \chi^A_1 X +\chi^A_0
\]
its characteristic polynomial.
Then $\chi_k^A$ for all $k$ are polynomials in the $a_{ij}$'s.

For any $(A,B)\in \scrZ_{p,m,r} =\ov{{\rm Im}(\phi_{p,m,r})}$ (see \eqref{eq:phi} and Proposition~\ref{prop:irred}),  we have
\begin{equation}
  \label{eq:rank}
 {\rm rank}(A)\leq 2p+m, \quad {\rm rank}(B)\leq 2p+r,
\end{equation}
and $\chi_{0}^A= \ldots =\chi_{r-1}^A=0 =\chi_0^B= \ldots =\chi_{m-1}^B$.
In particular, ${\rm rank}(A^2)\leq 2p+m$ and ${\rm rank}(B^2)\leq 2p+r$.


Consider the following conditions on $n\times n$ matrices $A, B$:
\begin{align}
&\, \rank(A)=2p+m, \text{ and } \rank(B)=2p+r,
 \label{eq1:U}\\
&\text{ both }  A \text{ and }B \text{ are diagonalizable, }
 \label{eq2:U}\\
&\text { all the nonzero eigenvalues of }A \text{ (respectively, } B  \text{) are distinct},
 \label{eq3:U}\\
&\text{ all the nonzero eigenvalues of }A^2 \text{ (respectively, } B^2  \text{) are distinct}.
 \label{eq4:U}
\end{align}
Introduce the following subset of $\scrZ_{p,m,r}$:
\begin{align}
  \label{eq:UU}
\calU_{p,m,r} =\{ (A,B) \in \scrZ_{p,m,r}^\circ \mid \text{ Conditions~\eqref{eq1:U}--\eqref{eq4:U} hold for } (A, B) \}.
\end{align}
The set $\calU_{p,m,r}$ has several favorable properties, and it will be used often in this paper.

\begin{lem}
  \label{lem:Upmr}
The subset $\calU_{p,m,r}$ of $\scrZ_{p,m,r}$ is open and nonempty.
Moreover, we have
\[
\calU_{p,m,r}\subseteq {\rm Im}(\phi_{p,m,r}),
\qquad
\ov{\calU_{p,m,r}} =\scrZ_{p,m,r}.
\]
\end{lem}

\begin{proof}
The first assertion follows by definition as each of \eqref{eq1:U}--\eqref{eq4:U} is clearly a nonempty open condition.
Since $\calU_{p,m,r}$ is open in $\scrZ_{p,m,r}$, it follows from the irreducibility of $\scrZ_{p,m,r}$ (see Proposition \ref{prop:irred}) that $\ov{\calU_{p,m,r}} =\scrZ_{p,m,r}$.

It remains to show that $\calU_{p,m,r}\subseteq {\rm Im}(\phi_{p,m,r}).$ For any $(A,B)\in \calU_{p,m,r}$, by Definition \ref{def:Zpmr} and \eqref{eq:UU}, $A$ has $2p+m$ distinct nonzero eigenvalues, $a_1,-a_1,\dots,a_p,-a_p,a_{p+1},\dots,a_{p+m}$, for some $a_1,\dots,a_{p+m}$ such that $a_i\neq \pm a_j$ for $i\neq j$. Similarly, $B$ has $2p+r$ distinct nonzero eigenvalues. Note that both $A$ and $B$ are diagonalizable, that is, there exists $g_1\in G$ such that
\[
g_1Ag_1^{-1}={\diag}(a_1,-a_1,\dots, a_p,-a_p,a_{p+1},\dots,a_{p+m}\underbrace{0,\dots,0}_r).
\]
It follows by Proposition~ \ref{ACmatrix} that
\[
g_1Bg_1^{-1}={\diag}(C_1,\dots,C_p,\underbrace{0,\dots,0}_{m},D),
\]
where $C_i= \begin{bmatrix} 0& b_i\\ c_i&0 \end{bmatrix}$ for $1\leq i\leq p$, and $D$ is arbitrary. As we know $B$ has $2p+r$ distinct nonzero eigenvalues, we have $b_ic_i\neq 0$ and $D$ has $r$ distinct nonzero eigenvalues. Thus there exists $P \in GL_r$ such that $PDP^{-1}=\diag ( b_{p+1},\dots ,b_{p+r} )$. Letting $g =\diag(I_{2p+m}, P)\cdot g_1$, we have
\begin{align}
gAg^{-1} &=A' \stackrel{\rm def}{=} {\diag}(a_1,-a_1,\dots, a_p,-a_p,a_{p+1},\dots,a_{p+m}, \underbrace{0,\dots,0}_r),
  \label{eq:gAg}\\
gBg^{-1} &= B' \stackrel{\rm def}{=} {\diag}(C_1,\dots,C_p,\underbrace{0,\dots,0}_{m}, b_{p+1},\dots ,b_{p+r} ),
 \label{eq:gBg}
\end{align}
where $C_i= \begin{bmatrix} 0& b_i\\ c_i&0 \end{bmatrix}$ for any $1\leq i\leq p$. Therefore,
we have $(gAg^{-1}, gBg^{-1}) \in {\rm Im}(\phi_{p,m,r})$ and then $(A,B) \in {\rm Im}(\phi_{p,m,r})$.
\end{proof}


\begin{prop}\label{prop:dimsion of Zprm}
For $(p,m,r) \in\TPL_n$, we have $\dim \scrZ_{p,m,r}=n^2+p$.
\end{prop}
\begin{proof}
It suffices to show that $\dim \calU_{p,m,r}=n^2+p$, since $\scrZ_{p,m,r}$ is irreducible.

Denote by $V=\phi_{p,m,r}^{-1}(\calU_{p,m,r})$, an open subset in $G\times \bbC^{p+m}\times \bbC^{2p+r}$, and then the restriction of $\phi_{p,m,r}$ to $V$ is also a regular map, which is also denoted by $\phi_{p,m,r}$. Clearly, $\calU_{p,m,r}$ and $V$ are irreducible varieties, and $\dim V=n^2+3p+m+r=n^2+n+p$.

Below we shall freely use the notations in the proof of Lemma \ref{lem:Upmr} above.

For any $(A,B) \in \calU_{p,m,r}$, we have $(A', B') \in \calU_{p,m,r}$, where $A'=gAg^{-1}, B'=gBg^{-1}$ are given in \eqref{eq:gAg}--\eqref{eq:gBg}.  Let $W_p =\bbZ_2^p\rtimes S_p$. By definition, up to a permutation action by the finite group $W_p \times S_m \times S_r$,  a point in $\phi_{p,m,r}^{-1}((A',B'))$ is  of the form
\[
\big(h, (a_1,\dots,a_p,a_{p+1},\dots,a_{p+m}), (d_1,e_1,\dots,d_p,e_p,d_{p+1},\dots,d_{p+r}) \big)\in G\times \bbC^{p+m}\times \bbC^{p+r},
\]
such that
\begin{align}
h\, {\diag}(a_1,-a_1,\dots,a_p,-a_p, a_{p+1},\dots,a_{p+m},\underbrace{0,\cdots,0}_{r})h^{-1}=A',
 \label{eq:A3} \\
h\, {\diag}(D_1,\dots,D_p, \underbrace{0,\dots,0}_{m},d_{p+1},\dots ,d_{p+r} )h^{-1}= B',
 \label{eq:B3}
\end{align}
where $D_i=\begin{bmatrix}0& d_i\\ e_i&0 \end{bmatrix}$ for $1\leq i\leq p$.
So from (\ref{eq:A3})--(\ref{eq:B3}), we have
$d_{p+i}=b_{p+i},\text{ for }1\leq i\leq r,$
and
$$h={\diag}(h_1,\dots,h_{2p},h_{2p+1}, \dots,h_{2p+r},h_{2p+r+1},\dots,h_{n} )$$
for any nonzero $h_i$ for $1\leq i\leq n$.
Furthermore, it follows from (\ref{eq:B3}) that
\begin{equation*}
h_{2i-1}d_i =b_ih_{2i}, \qquad
h_{2i}e_i = h_{2i-1}c_i,
\end{equation*}
for $1\leq i\leq p$.
Then
any point in $\phi_{p,m,r}^{-1}((A',B'))$ is of form
$$(h, (a_1,\dots,a_p,a_{p+1},\dots,a_{p+m}), (u_1b_1, u_1^{-1}{c_1},\dots,u_pb_p, u_p^{-1} {c_p}, b_{p+1},\dots,b_{p+r})),$$
where $h={\diag}(h_1,u_1h_1,\dots,h_{2p-1},u_ph_{2p-1},h_{2p+1},\dots, h_{n})$, for nonzero scalars $u_1$, $\dots$, $u_p$, $h_1$, $h_3$, $\dots$, $h_{2p-1}$, $h_{2p+1}$, $\dots$, $h_n$.
So $\dim \phi_{p,m,r}^{-1}((A',B'))=n$.
By a standard result in algebraic geometry, we have
$\dim \calU_{p,m,r}=\dim V-\dim \phi_{p,m,r}^{-1}((A',B'))=n^2+p$.
\end{proof}

\subsection{Some useful lemmas}

The following two lemmas will be useful in later sections.
\begin{lem}
  \label{lem:ZZZ12}
For $(A_1,B_1)\in\scrZ_{p_1,m_1,r_1}$ and  $(A_2,B_2)\in\scrZ_{p_2,m_2,r_2}$, we denote
$A=\diag(A_1,A_2)$, and $B=\diag(B_1,B_2)$. Then we have $(A,B)\in\scrZ_{p_1+p_2,m_1+m_2,r_1+r_2}$.
\end{lem}

\begin{proof}
Set $p=p_1+p_2$, $m=m_1+m_2$, $r=r_1+r_2$. 

%

There is a natural embedding $\theta: \scrZ_{p_1,m_1,r_1}\times \scrZ_{p_2,m_2,r_2} \hookrightarrow \scrZ_{2p+m+r}$ given by
\[
\theta\big((A_1,B_1),(A_2,B_2)\big)= \big(\diag(A_1,A_2),\diag(B_1,B_2) \big).
\]
Let
\[
\calU'= \big\{\big((A_1,B_1),(A_2,B_2)\big)\in \calU_{p_1,m_1,r_1}\times \calU_{p_2,m_2,r_2} \mid \theta\big((A_1,B_1),(A_2,B_2)\big)\in \calU_{p,m,r} \big\}.
\]

For any $\big((A_1,B_1),(A_2,B_2)\big)\in \calU_{p_1,m_1,r_1}\times \calU_{p_2,m_2,r_2}$, without loss of generality, we assume $A_i$ $(i=1,2)$ are diagonal matrices, i.e.,
\begin{align*}
A_1&=\diag(a_1, -a_1, a_2, -a_2, \ldots, a_{p_1}, -a_{p_1},  a_{p_1+1}, \ldots, a_{p_1+m_1},\underbrace{0,\ldots, 0}_{r_1}),
\\
A_2 &=\diag(b_1, -b_1, b_2, -b_2, \ldots, b_{p_2}, -b_{p_2},  b_{p_2+1}, \ldots, b_{p_2+m_2},\underbrace{0,\ldots, 0}_{r_2}).
\end{align*}
Let $A_2(t)=\diag(b_1+t, -b_1-t, \ldots, b_{p_2}+t, -b_{p_2}-t,  b_{p_2+1}+t, \ldots, b_{p_2+m_2}+t,\underbrace{0,\ldots, 0}_{r_2})$.
Then $\big((A_1,B_1),(A_2(t),B_2)\big)\in \calU'$ for $t\in \bbC^*$ close to $0$. And when $t$ approaches $0$, the limit $\big((A_1,B_1),(A_2(t),B_2)\big)$ is $\big((A_1,B_1),(A_2,B_2)\big)$.
So $\calU_{p_1,m_1,r_1}\times \calU_{p_2,m_2,r_2}\subseteq \ov{\calU'}$. It follows that $\overline{\calU_{p_1,m_1,r_1}\times \calU_{p_2,m_2,r_2} }= \ov{\calU'}$ since $\calU_{p_1,m_1,r_1}\times \calU_{p_2,m_2,r_2}\supseteq \calU'$. Now, using the irreducibility of $\scrZ_{p_i,m_i,r_i}$ from Proposition~\ref{prop:irred}, we see that
\[
\ov{\calU'}=\ov{\calU_{p_1,m_1,r_1}\times \calU_{p_2,m_2,r_2}}  
=\scrZ_{p_1,m_1,r_1}\times \scrZ_{p_2,m_2,r_2}.
\]
Therefore, by a standard fact in topology we have
$\theta(\scrZ_{p_1,m_1,r_1}\times \scrZ_{p_2,m_2,r_2})
\subseteq \ov{\theta(\calU')}
\subseteq \scrZ_{p,m,r}$.
The lemma is proved.
\end{proof}

\begin{lem}
  \label{lem:switch}
Let $(p,m,r) \in \TPL_n$.   We have an isomorphism of varieties $\scrZ_{p,m,r}\rightarrow \scrZ_{p,r,m}$, which sends $(A,B)$ to $(B,A)$.
\end{lem}

\begin{proof}


Define an injective morphism $\theta:\scrZ_{p,m,r}\rightarrow \scrZ_{2p+m+r}$ by $\theta(A,B)=(B,A)$. We have $\calU_{p,m,r} =  \theta (\calU_{p,r,m})$ by definition; cf. \eqref{eq:UU}. Therefore, we obtain a morphism $\theta: \scrZ_{p,m,r}\rightarrow \scrZ_{p,r,m}$, $(A,B) \mapsto (B,A)$. By switching $m$ and $r$, we obtain a morphism $\theta': \scrZ_{p,r,m}\rightarrow \scrZ_{p,m,r}$,  inverse to $\theta$.
\end{proof}

\section{The variety $\scrZ_n$ is covered by the subvarieties $\scrZ_{p,m,r}$}
  \label{sec:cover}

In this section, we shall prove that $\scrZ_n$ is the union of its subvarieties $\scrZ_{p,m,r}$, where the triples $(p,m,r)$ run over the indexing set $\TPL_n$ \eqref{TPL}.

\subsection{$\scrR_A$ for $A$ with eigenvalues $\pm a \in \bbC^*$}

We start with $A$ being diagonalizable.

\begin{lem}\label{lem: diagonlised matrix}
Let $A=\diag(\underbrace{a,\ldots,a}_s,\underbrace{-a\ldots,-a}_t)$ with $a\in \bbC^*$ and $n=s+t$.
Then for any $B\in \scrR_A$, the following hold:
\begin{enumerate}
\item The rank of $B$ is at most $2\min(s,t)$;

\item  $(A,B)\in \scrZ_{\min(s,t),|s-t|,0}$.
\end{enumerate}
\end{lem}

\begin{proof}
Without loss of generality, we assume that $s\geq t$.

(1). Any $B\in \scrR_A$ can be written  as
\begin{align*}
B =
\begin{bmatrix}
0   & B'\\
B'' &0
\end{bmatrix}
\end{align*}
for some matrices $B'\in M_{s\times t}(\bbC)$ and $B''\in M_{t\times s}(\bbC)$ of the form
as in Lemma \ref{ACmatrix}. Thus the rank of $B$ is at most $2t$ for $t\leq s$.

(2). There exists an invertible matrix $P$ such that
\[
A'=P^{-1}AP=\diag(\underbrace{a,-a,a,-a,\ldots,a,-a}_t,\underbrace{a,\ldots,a}_{s-t}).
\]
It suffices to show $(A',B)\in \scrZ_{t,s-t,0}$ for any $B\in \scrR_{A'}$.

Let $U$ be the subset of $\scrR_{A'}$ which consists of all matrices $B \in \scrR_{A'}$ which have $2t$ distinct nonzero eigenvalues. Then $U$ is a nonempty open set of $\scrR_{A'}$ since
\[
U':= \left \{\diag \Big(
\begin{bmatrix}
  0   & b_1    \\
c_1  & 0
 \end{bmatrix},\ldots,\begin{bmatrix}
  0   & b_t    \\
c_t & 0
 \end{bmatrix},
 \underbrace{0,\ldots,0}_{s-t}
\Big) \mid  b_ic_i\neq b_jc_j \  \text{and} \   b_ic_i\in \bbC^*, \ \forall i\neq j \right\}\subseteq U.
\]
(Actually, one can further show that $(A',B)$ for any $B\in U$ is similar to $(A',B')$ for some $B' \in U'$. We will not use this fact below.)

Assume $\alpha $ is a nonzero eigenvalue of $B\in U$ with an eigenvector $\xi$ such that $B\xi=\alpha\xi$.
Then $B(A'\xi)=-A'(B\xi)=-A'(\alpha\xi)=-\alpha(A'\xi)$,
which implies that $-\alpha$ is an eigenvalue of $B$ (note $A'\xi \neq 0$ since $A'$ is invertible).
Hence $\alpha$ is a nonzero eigenvalue of $B$ if and only if so is $-\alpha$, i.e.,
the nonzero eigenvalues of $B$ appears in pairs.
Hence, we have $(B,A')\in \scrZ_{t,0,s-t}$. By Lemma \ref{lem:switch} we have $(A',B)\in \scrZ_{t,s-t,0}$ for any $B\in U$, and then for any $B\in \ov{U} =\scrR_{A'}$ (since $\scrR_{A'}$ is a vector space).
This completes the proof.
\end{proof}

For any vector $\vec{a}=(a_1,\dots,a_n)$,  let $J_n(\vec{a})$ denote the following matrix
$$
J_n(\vec{a})=\left[\begin{array}{ccccccc}
 a_1     &    1                                           \\
         &a_2    &1                                        \\
         &       &   \ddots &  \ddots                      \\
         &       &          &\ddots      &  1             \\
         &       &          &            &a_n
\end{array}\right].
$$
In particular, if $\vec{a}=(\underbrace{a,\dots,a}_n)$,
then $J_n(\vec{a}) =J_n(a)$ is the $n\times n$ Jordan block of eigenvalue $a$.

The next two lemmas, Lemmas~\ref{lem:AC2dist} and \ref{lem:AC2a}, are generalizations of Lemma~\ref{Jordan anti-layer}. They will be used (only) in the proof of Proposition~\ref{prop:block a,-a} below.

\begin{lem}\label{lem:AC2dist}
For any $m,n\geq 1$,
let $a_1,\ldots,a_m,a_{m+1},\ldots,a_{m+n}\in \bbC^*$ with $a_j\neq -a_i$ for any $i, j$.
For any $\vec{a}=(a_1,\dots, a_{m})$ and $\vec{b}=(a_{m+1},\dots,a_{m+n})$,
the matrix $B=(b_{ij})_{m\times n}$ satisfies that $J_m(\vec{a})B+BJ_n(\vec{b})=0$ if and only if
$B=0$.
\end{lem}

\begin{proof}
Let $B=(b_{ij})$ be an arbitrary $m\times n$  matrix such that
\begin{align*}
&0=J_m(\vec{a})B+BJ_n(\vec{b})=
\\
&\begin{tiny}\begin{bmatrix}
(a_1+a_{m+1})b_{1,1}+b_{2,1}                 &  (a_1+a_{m+2})b_{1,2}+b_{2,2}+b_{1,1}                  &    ...       & (a_1+a_{m+n})b_{1,n}+b_{2,n}+b_{1,n-1}              \\
(a_2+a_{m+1})b_{2,1}+b_{3,1}                  &   (a_2+a_{m+2})b_{1,2}+b_{3,2}+b_{2,1}                 &    ...       & (a_2+a_{m+n})b_{2,n}+b_{3,n}+b_{2,n-1}             \\
\vdots                                       &  \vdots                                                &              &   \vdots            \\
(a_{m-1}+a_{m+1})b_{m-1,1}+b_{m,1}           &  (a_{m-1}+a_{m+2})b_{m-1,2}+b_{m,2}+b_{m-1,1}          &    ...       & (a_{m-1}+a_{m+n})b_{m-1,n}+b_{m,n}+b_{m-1,n-1}    \\
(a_m+a_{m+1})b_{m,1}                         & (a_m+a_{m+2})b_{m,2} +b_{m,1}                          &    ...       & (a_m+a_{m+n})b_{m,n} +b_{m,n-1}
\end{bmatrix}
\end{tiny}.
\end{align*}
Note by assumption that $a_i \neq -a_j$ for all $i,j$. Then the equation above at the $(m,1)$th entry implies that $b_{m,1}=0$. Then an inspection of the above matrix equation at the first column and the last row shows all elements in the first column and the last row of $B$ are zeros. Repeating the process for the second column and second last row starting at $(m-1,2)$th entry, we  see that all elements of the second column and the second last row of $B$ are zeros. Continuing the same process, we conclude that $B=0$.
\end{proof}

\begin{lem}\label{lem:AC2a}
Let $m,n\geq 1$. Let $a_1,\dots,a_{\max (m,n)}\in \bbC^*$ with $a_j\neq \pm a_i$ for all $i\neq j$,
$\vec{a}=(a_1,\dots, a_{m})$ and $\vec{a}'=(-a_1,\dots,-a_{n})$.
Then an $m\times n$ matrix $B=(b_{ij})$ satisfies $J_m(\vec{a})B+BJ_n(\vec{a}')=0$ if and only if
$B$ is one of the two forms \eqref{eq:B1}--\eqref{eq:B2}:
\begin{eqnarray}
  \label{eq:B1}
B=\left[\begin{array}{cccccccccc}
 b_{1,1}    & b_{1,2}          & \ldots                 &      b_{1,n-1}              &  b_{1,n}      \\
   0        & b_{2,2}          & \ldots                 &      b_{2,n-1}              &  b_{2,n}      \\
   0        & 0                & \ddots                 &      b_{3,n-1}              &  b_{3,n}      \\
\vdots      &                  &  \ddots                &      \ddots                 &               \\
0           &  0               &                        & 0                           &  b_{n,n}       \\
0           &  0               & \ldots                 &    0                        &  0            \\
\vdots      & \vdots           &                        &                             & \vdots        \\
0           &  0               & \ldots                 &                             &  0
 \end{array} \right]
 \quad \text{if }  m\geq n,
\end{eqnarray}
where $b_{i,j}=-b_{i-1,j-1}-(a_{i-1}-a_j)b_{i-1,j}$  for $2\leq i\leq j\leq n$
and $b_{1,j}$ are arbitrary for $1\leq j\leq n$;
or
\begin{eqnarray}
  \label{eq:B2}
B=\left[\begin{array}{cccccccccc}
b_{1,1} &  b_{1,2} &  \ldots  &  b_{1,m-1} & b_{1,m}    &  b_{1,m+1}    & \ldots        &  b_{1,n}      \\
  0     &  b_{2,2} &  \ldots  &  b_{2,m-1} & b_{2,m}    & b_{2,m+1}     & \ldots        &  b_{2,n}    \\
0       &  0       &  \ddots  &  b_{3,m-1} & b_{3,m}    & b_{3,m+1}     & \ldots        &   b_{3,n} \\
        &          &  \ddots  &  \ddots    & \vdots    &  \vdots        &               & \vdots\\
0       &  0       &          &     0      & b_{m,m}    & b_{m,m+1}     & \ldots        &   b_{m,n}
\end{array} \right]
\quad \text{if }  m\leq n,
\end{eqnarray}
where $b_{1,j}$ are arbitrary for $n-m+1\le j \le n$ and
\begin{eqnarray*}
  \begin{cases}
  b_{i,n-m+j}=-b_{i-1,n-m+j-1}-(a_{i-1}-a_{n-m+j})b_{i-1,n-m+j}       &  \text{if} \  2\leq i\leq j\leq m, \\
  b_{m,j}=(a_{j+1}-a_{m})b_{m,j+1}                                    &  \text{if} \  m\leq j\leq n,       \\
  b_{i,j}=-b_{i+1,j+1}-(a_{i}-a_{j+1})b_{i,j},                        &  \text{if} \   1\leq i\leq j\leq n-m-1+i.
  \end{cases}
\end{eqnarray*}
\end{lem}

\begin{proof}
Let $B=(B_{ij})$ be an arbitrary $m\times n$  matrix such that
\begin{align*}
&0=J_m(\vec{a})B+BJ_n(\vec{a}')=
\\
&{\tiny
\begin{bmatrix}
(a_1-a_1)b_{1,1}+b_{2,1}                 &  (a_1-a_2)b_{1,2}+b_{2,2}+b_{1,1}                  &    ...       & (a_1-a_n)b_{1,n}+b_{2,n}+b_{1,n-1}              \\
(a_2-a_1)b_{21}+b_{3,1}                  &   (a_2-a_2)b_{1,2}+b_{3,2}+b_{2,1}                 &    ...       & (a_2-a_n)b_{2,n}+b_{3,n}+b_{2,n-1}             \\
\vdots                                   &  \vdots                                            &              &   \vdots            \\
(a_{m-1}-a_1)b_{m-1,1}+b_{m,1}           &  (a_{m-1}-a_2)b_{m-1,2}+b_{m,2}+b_{m-1,1}          &    ...       & (a_{m-1}-a_n)b_{m-1,n}+b_{m,n}+b_{m-1,n-1}    \\
(a_m-a_1)b_{m,1}                         & (a_m-a_2)b_{m,2} +b_{m,1}                          &    ...       & (a_m-a_n)b_{m,n} +b_{m,n-1}
 \end{bmatrix}.
 }
\end{align*}
The lemma follows.
\end{proof}

\begin{rem}  \label{rem:AC2a}
Lemma \ref{lem:AC2a} reduces to Lemma~\ref{Jordan anti-layer} and $B$ becomes an anti-layered matrix if we allow $a_i=a \in \bbC^*$ for all $i$.
\end{rem}

For a partition $\nu =(\nu_1, \nu_2, \ldots)$, we denote by $\nu' =(\nu'_1, \nu'_2, \ldots)$ the transposed partition. We also denote by $J_\nu(a)$ the Jordan normal form of eigenvalue $a$ and of block sizes $\nu_1, \nu_2, \cdots$. The following result generalizes Lemma~\ref{lem: diagonlised matrix}.

\begin{prop}
   \label{prop:block a,-a}
 Let $A=\diag(J_{\nu}(a), J_{\mu}(-a))$, for $a\in\bbC^*$, a partition $\nu$ of $n$ and a partition $\mu$ of $m$. Denote $p=\sum_{i\ge 1} \min(\nu'_i,\mu'_i).$
Then we have $(A,B)\in \scrZ_{p,n+m-2p,0}$,  for all $B\in \scrR_A$.
\end{prop}

\begin{proof}
Let us write the partitions in parts as $\nu=(n_1, \ldots, n_r)$ and $\mu =(m_1,\ldots, m_s)$.

By Lemma~\ref{Jordan anti-layer}, any $B \in \scrR_{A}$ can be written in the same block shape as for $A$, $B=(B_{ij})$, such that $B_{ij} =0$, for $1\leq i,j\leq r$ or $r+1 \leq i,j \leq r+s$, while $B_{ij}$ and $B_{ji}$ are anti-layered matrices, for all $1\leq i\leq r< j\leq r+s$.

Set $N=\max (n_1,\dots,n_r,m_1,\dots,m_s)$.
Let $a_1,\dots,a_N\in\bbC^*$ such that $a_i\neq \pm a_j$ for any $j\neq i$.
Consider the diagonalizable matrix
\[
A'=\diag(A_1,\dots, A_r,A_{r+1},\dots,A_{r+s}),
\]
where
\[
A_i=J_{n_i}(\vec{\alpha}_i),\,\, \forall 1\leq i\leq r;
\quad A_j= J_{m_{j-r}}(\vec{\beta}_j), \,\,\forall r+1\leq j\leq r+s,
\]
and $\vec{\alpha}_i=(a_1,\dots,a_{n_i})$, $\vec{\beta}_j=(-a_1,\dots,-a_{m_j})$. Let $B'=(B_{ij})$ be of the same block shape as for $A'$.
Then $B' \in \scrR_{A'}$ if and only if
\[
A_iB'_{ij}+B'_{ij}A_j=0,\quad \forall 1\leq i,j\leq r+s.
\]
It follows by Lemma~ \ref{lem:AC2dist} that $B'_{ij}=0$  for all $1\leq i,j\leq r$ or $r+1\leq i,j\leq r+s$.
Lemma ~\ref{lem:AC2a} describes the detailed structures of $B'_{ij}$ and $B'_{ji}$ for all $1\leq i\leq r< j\leq r+s$.
Therefore, by applying Remark~ \ref{rem:AC2a} and the observation in the preceding paragraph, for any $B=(B_{ij})\in\scrR_A$, there exists $B'=(B'_{ij})\in \scrR_{A'}$ such that $B'$ approaches $B$ as $a_i\rightarrow a$ (e.g., by taking $a_i=a+it$ with $t\rightarrow 0$). This reduces the proof of the proposition to the following.

\vspace{2mm}
\noindent{\bf Claim.} We have $(A',B') \in\scrZ_{p,n+m-2p,0}$.

Let us prove the Claim. Since $A'$ is similar to the matrix
$$
\diag(a_1,\dots, a_{n_1},\dots, a_1,\dots,a_{n_r}, -a_1,\dots,-a_{m_1},\dots, -a_1,\dots,-a_{m_s} ),
$$
there exists $g\in \GL_{n+m}$ such that
\begin{eqnarray*}
g^{-1}A'g=\diag(A''_1,\dots,A''_N),
\end{eqnarray*}
where $A''_i= \diag(\underbrace{a_i,\dots,a_i}_{\nu'_i},\underbrace{-a_i,\dots,-a_i}_{\mu'_i} )$ for each $1\leq i\leq N$.
Applying Lemma \ref{ACmatrix}, we see that $g^{-1}B'g \in \scrR_{g^{-1}A'g}$ must be of the form
\[
g^{-1}B'g =\diag(B''_{1},\dots, B''_{N} )
\]
where $B''_{i}$ is a matrix of size $(\nu'_i+\mu'_i)\times (\nu'_i+\mu'_i)$ for $1\leq i\leq N$. We have $(A''_i,B''_i)\in \scrZ_{\min(\nu'_i, \mu'_i),|\nu'_i- \mu'_i|,0}$ by Lemma~ \ref{lem: diagonlised matrix}. From Lemma~ \ref{lem:ZZZ12} we conclude that $(g^{-1}A'g,g^{-1}B'g)$ is in $\scrZ_{p,n+m-2p,0}$, and then so is $(A',B')$.
The proposition is proved.
\end{proof}

\subsection{$\scrR_A$ for $A$ nilpotent with Jordan blocks of same size}


Recall the floor (resp., ceil) function $\lfloor{x}\rfloor$ (resp., $\lceil{x}\rceil$) which denotes
the largest (resp., least) integer $\leq x$ (resp., $\geq x$) for $x\in \mathbb R$. Recall $\mathbb{J}_{s,n}(0)$ from \S\ref{subsec:block}.

\begin{prop}\label{prop:RJsn}
For any $B\in\scrR_{\mathbb{J}_{s,n}(0)}$, we have $(\mathbb{J}_{s,n}(0),B)\in  \scrZ_{n\lfloor \frac{s}{2}\rfloor,0,n(\lceil\frac{s}{2}\rceil-\lfloor \frac{s}{2}\rfloor)}$. Moreover, there exists a dense open subset $\calU_{\mathbb{J}_{s,n}(0)}$ of  $\scrR_{\mathbb{J}_{s,n}(0)}$ which consists of matrices similar to
$\diag \big(J_{\lceil \frac{s}{2}\rceil}(b_1),J_{\lfloor \frac{s}{2}\rfloor}(-b_1), \ldots, J_{\lceil \frac{s}{2}\rceil}(b_{n}),J_{\lfloor \frac{s}{2}\rfloor}(-b_{n}) \big),$
for distinct $b_1, \ldots, b_n \in \bbC^*$.
\end{prop}

\begin{proof}
The case for $s=1$ is clear as $\mathbb{J}_{1,n}(0)$ is the $n\times n$ zero matrix.

We shall now assume $s\ge 2$. Set $B=\ALB (s,s,B_1,\ldots, B_s)$, cf. \eqref{eq:block al1} with $\sigma=-$ for notation.
We introduce the following variety
\begin{align*}
\calV_{s,n} &=\big\{ (\la, B_1,\ldots, B_s) \in \bbC \times M_n(\bbC)^s \mid
\\
&\qquad\qquad
\text{ all } (sn-1)\times (sn-1) \text{ minors of } (\la I_{sn} - B ) \text{ are } 0 \big\}.
\end{align*}
Let us denote $B_i=(b_{i;k,\ell})_{k,\ell}$, for each $i$. Then $F:=\det (\la I_n -B_1) \det (\la I_n +B_1)$ is a polynomial in the $b_{1;k,\ell}$'s. Let $C_{sn,1}$ be the matrix obtained from $\la I_{sn} - B$ with the first column and the last row deleted, and denote $G =\det C_{sn,1}$ (which is an $(sn-1)\times (sn-1)$-minor of $\la I_{sn} - B$).
Observe that the product of the diagonal entries of $C_{sn,1}$ gives us (up to a sign) a monomial $b_{2;n,1}^{s-1}\prod_{i=1}^{n-1} b_{1;i,i+1}^s,$ which is of highest degree in $b_{2;n,1}$ and is not equal to any other monomials from the expansion of $G=\det C_{sn,1}$. It follows that $(F,G)$ is a regular sequence.

Then
\[
\calV_{s,n} \subset \calV_{s,n}' := \big\{ (\la, B_1,\ldots, B_s) \in \bbC \times M_n(\bbC)^s \mid
F=G=0 \big\}.
\]
It follows that $\dim \calV_{s,n} \le \dim \calV_{s,n}'  =(1+sn^2)-2=sn^2-1$.

Consider the composition of a projection with an isomorphism
\[
\psi: \calV_{s,n} \longrightarrow M_n(\bbC)^s \stackrel{\cong}{\longrightarrow} \scrR_{\bbJ_{s,n}(0)},
\quad  (\la, B_1,\ldots, B_s) \mapsto \ALB (s,s,B_1,\ldots, B_s).
\]
Denote by $W^c$ the complement of a subvariety $W$ in $\scrR_{\bbJ_{s,n}(0)}$.
By definition, we have
\begin{equation*}
\text{Im}(\psi)^c =\big\{B= \ALB (s,s,B_1,\ldots, B_s) ) \mid \rank (\la I_{sn} -B) \ge sn-1, \forall \la \in \bbC \big\}.
\end{equation*}
Clearly $U_1:=\big\{B= \ALB (s,s,B_1,\ldots, B_s) ) \mid B_1 \text{ has distinct eigenvalues } b_1, \ldots, b_n \big\}$ is a dense open subset of $\scrR_{\bbJ_{s,n}(0)}$.
Since $\dim \ov{\text{Im}(\psi)} = \dim \text{Im}(\psi) \le \dim \calV_{s,n} \le sn^2-1 <\dim \scrR_{\bbJ_{s,n}(0)}$, we see
\[
\calU_{\mathbb{J}_{s,n}(0)} := \ov{\text{Im}(\psi)}^c \cap U_1
\]
is dense open in $\scrR_{\bbJ_{s,n}(0)}$. By construction we have
\begin{align}
\label{eq:Ubig}
\begin{split}
\calU_{\mathbb{J}_{s,n}(0)} &\subseteq \big\{B= \ALB (s,s,B_1,\ldots, B_s) ) \mid \rank (\la I_{sn} -B) \ge sn-1, \forall \la \in \bbC,
\\
&\qquad\qquad\qquad\qquad\qquad\qquad\qquad  B_1 \text{ has distinct eigenvalues } b_1, \ldots, b_n \big\}.
\end{split}
\end{align}

It follows by \eqref{eq:Ubig} that for any $B\in \calU_{\mathbb{J}_{s,n}(0)}$ there is a unique Jordan block with eigenvalue $b_i$ (respectively, $-b_i$) in the Jordan normal form  of $B$, which is $J_{\lceil\frac{s}{2}\rceil}(b_i)$ (respectively, $J_{\lfloor\frac{s}{2}\rfloor}(b_i)$), for all $i$. This proves the second statement in Proposition~\ref{prop:RJsn}. The first statement follows from Lemma~\ref{lem:switch}, the second statement, Proposition~\ref{prop:block a,-a}, and that $\ov{\calU_{\mathbb{J}_{s,n}(0)}} =\scrR_{\bbJ_{s,n}(0)}$.
\end{proof}

\begin{rem}
If one could show that $\text{Im}(\psi)$ is closed, then $\calU_{\mathbb{J}_{s,n}(0)}$ would be equal to the RHS of \eqref{eq:Ubig}; this would make the set $\calU_{\mathbb{J}_{s,n}(0)}$ in Proposition~\ref{prop:RJsn} above explicit.
\end{rem}

\subsection{$\scrR_A$ for $A$ nilpotent}

Let $A=\diag(\mathbb{J}_{1,n_1}(0), \mathbb{J}_{2,n_2}(0),\dots,\mathbb{J}_{s,n_s}(0))$, and $n=\sum_{i=1}^s in_i$. Then each $B\in \scrR_{A}$ is of block matrix form $B=(B_{ij})_{s\times s}$ such that
$\mathbb{J}_{i,n_i}(0) B_{ij}=-B_{ij} \mathbb{J}_{j,n_j}(0)$, for $1\leq i,j\leq s$. By Lemma~\ref{lem:ALB}, we have
\begin{equation}
  \label{eq:Bijk}
B_{ij}=\ALB \big(i,j, B^{(1)}_{ij}, B^{(2)}_{ij},\ldots,B^{(\min(i,j))}_{ij} \big)
\end{equation}
for $n_i\times n_j$ matrices ${B^{(1)}_{ij}}, {B^{(2)}_{ij}},\ldots, B^{(\min(i,j))}_{ij}$, and $1\leq i, j\leq s$.
We shall also view $B$ as a $\frac12 s(s+1) \times \frac12 s(s+1)$ block matrix with diagonal blocks being $\pm B_{ii}^{(1)}$ (of multiplicities $\lceil \frac{i}{2}\rceil$ and ${\lfloor\frac{i}{2}\rfloor}$, respectively) for $1\le i \le s$.

\begin{lem}\label{lem:detblock}
Let $A=\diag(\mathbb{J}_{1,n_1}(0), \mathbb{J}_{2,n_2}(0),\dots,\mathbb{J}_{s,n_s}(0))$, and retain the notation above for $B\in \scrR_{A}$; see \eqref{eq:Bijk}. Then
\begin{itemize}
\item[(1)] $B$ is similar to a $\frac12 s(s+1) \times \frac12 s(s+1)$ block upper-diagonal matrix with diagonals being a rearrangement of the diagonal blocks of $B$; see \eqref{eq:PBP} and \eqref{eq:new diag} below;
\item[(2)] $\chi^{B}(X)=\prod_{i=1}^s\chi^{B_{ii}}(X)$;
\item[(3)] $\chi^B(X)=\prod_{i=1}^s ({\chi^{B^{(1)}_{ii}}}(X))^{\lceil \frac{i}{2}\rceil}({\chi^{{-B^{(1)}_{ii}}}}(X))^{\lfloor\frac{i}{2}\rfloor}$.
\end{itemize}
\end{lem}

\begin{proof}
Parts (2)-(3) follow by (1) immediately. It remains to prove (1). Let us write down the detail for $s=3$ to illustrate the general idea.

We start with a simple linear algebra fact, which will be applied repeatedly. Consider a distinguished block matrix (in 4 blocks)
of the form
\begin{align}
  \label{eq:C}
C=
\begin{bmatrix}
C_{11} & C_{12}  & C_{13}    & C_{14} \\
0          & C_{22}  &  0            & C_{24}\\
C_{31}  &C_{32}   & C_{33}   & C_{34} \\
0           &0            & 0            & C_{44}
\end{bmatrix}.
\end{align}
Assume the block size of $C$ is $r|s|t|u$. The by conjugation a block permutation matrix $P_1$ transforms $C$ to a block upper-triangular matrix as follows:

\begin{align}
    \label{eq:PCP}
P_1=
\begin{bmatrix}
I_{rr} & 0         & 0     & 0      \\
0      & 0         & I_{ss} & 0     \\
0      &I_{tt}  & 0 & 0     \\
0      &  0     & 0 & I_{u}     \\
\end{bmatrix},
\qquad
P_1^{-1}CP_1=
\begin{bmatrix}
C_{11}  & C_{13}   & C_{12}    & C_{14} \\
C_{31}  &C_{33}    & C_{32}    & C_{34} \\
0           &         0    & C_{22}    & C_{24} \\
0           &         0    & 0             & C_{44}
\end{bmatrix}.
\end{align}

Now returning to our setting for $s=3$, we have
\begin{align*}
B=
\begin{bmatrix}
B^{(1)}_{11}   & 0                 &B^{(1)}_{12}      & 0                     &   0                    &B^{(2)}_{13} \\
B^{(1)}_{21}   &B^{(1)}_{22}  &  B^{(2)}_{22}   & 0                     &B^{(1)}_{23}     &B^{(2)}_{23} \\
0                    & 0                 &-B^{(1)}_{22}      & 0                     &0                      &-B^{(1)}_{23}  \\
B^{(1)}_{31}   &B^{(1)}_{32} &B^{(2)}_{32}      &B^{(1)}_{33}     &B^{(2)}_{33}    &B^{(3)}_{33} \\
0                   &0                   &-B^{(1)}_{32}      & 0                     &-B^{(1)}_{33}   &-B^{(2)}_{33} \\
0                   &0               &  0                          & 0                      &0                      &B^{(1)}_{33}    \\
\end{bmatrix},
\quad
P=
\begin{bmatrix}
0   & 0                 &0      & I                     &   0                    &0 \\
0   &I  &  0   & 0                     &0    &0 \\
0                    & 0                 &0      & 0                     &I                      &0 \\
I   &0 &0     &0     &0  & 0 \\
0                   &0                   &I      & 0                     &0  &0 \\
0                   &0               &  0                          & 0                      &0                      &I   \\
\end{bmatrix}.
\end{align*}
We view the matrix $B$ as a 4-block matrix $C$ in \eqref{eq:C}, by regarding the third diagonal entry $-B^{(1)}_{22}$ of $B$ as $C_{22}$ and the sixth diagonal entry $B^{(1)}_{33}$ of $B$ as the $C_{44}$.
Deleting the third and sixth rows and columns of $B$ we obtain a new matrix denoted by $B'_1$.
The above general discussion leading to \eqref{eq:PCP} shows that $B$ is similar to
$
B'= \begin{bmatrix}
B'_1 &  *            &   * \\
    & -B^{(1)}_{22} & \star \\
    &               & B^{(1)}_{33}
\end{bmatrix}.$

Note the first row of $B'_1$ is $0$ except the block $B^{(1)}_{11}$. We view $B'$ as a degenerate case of \eqref{eq:C} with $C_{11}$ being a $0\times 0$ matrix, and identify $B^{(1)}_{11}$ with $C_{22}$, and $C_{44}$ with $\begin{bmatrix}
    -B^{(1)}_{22} & \star \\
    & B^{(1)}_{33}
\end{bmatrix}.$ Applying the general discussion leading to \eqref{eq:PCP} shows that $B$ is similar to
\[
B''= \begin{bmatrix}
B''_1 & * &* & *  \\
 & B^{(1)}_{11} &  *            &   * \\
&    & -B^{(1)}_{22} & \star \\
 &   &               & B^{(1)}_{33}
\end{bmatrix},
\]
where $B''_1$ is obtained from $B'_1$ with the first row/column deleted.
Repeating the above process,
we conclude that $B$ is similar to
\begin{align}
  \label{eq:PBP}
P^{-1} B P =
\begin{bmatrix}
B^{(1)}_{33} & *            &      *       & *             &     *         & *           \\
             & B^{(1)}_{22} &      *       & *             &     *         & *           \\
             &              &-B^{(1)}_{33} & *             &  *            &      *        \\
             &              &              &B^{(1)}_{11}   &   *           &     *         \\
             &              &              &               & -B^{(1)}_{22} &    *         \\
             &              &              &               &               & B^{(1)}_{33}  \\
\end{bmatrix}.
\end{align}

The argument for general $s$ is entirely similar but involves messy notations. we skip the details, except noting that the diagonal entries in the resulting block upper-triangular matrix similar to $B$ are
\begin{equation}
  \label{eq:new diag}
\diag \big(B^{(1)}_{ss}; B^{(1)}_{s-1,s-1}, -B^{(1)}_{ss}; B^{(1)}_{s-2,s-2}, -B^{(1)}_{s-1,s-1}, B^{(1)}_{ss}; \ldots, \ldots; B^{(1)}_{11}, -B^{(1)}_{22}, \ldots, (-1)^{s-1} B^{(1)}_{ss} \big).
\end{equation}
The lemma is proved.
\end{proof}

We have the following generalization of Proposition~\ref{prop:RJsn}. Recall from Lemma~\ref{lem:sJordan} that an arbitrary nilpotent Jordan normal form $\diag(J_1^{n_1}(0), J_2^{n_2}(0),\ldots ,J_s^{n_s}(0))$ is similar to the matrix $\diag(\mathbb{J}_{1,n_1}(0),\ldots, \mathbb{J}_{s,n_s}(0))$.

\begin{prop}\label{prop:Jordan block 0}
Let $A=\diag(\mathbb{J}_{1,n_1}(0),\ldots, \mathbb{J}_{s,n_s}(0))$, $n=\sum^s_{i=1} in_i$, and $p=\sum^s_{i=1}n_i\lfloor{\frac{i}{2}}\rfloor$. Then for any $B\in \scrR_A$, we have $(A,B)\in\scrZ_{p,0,n-2p}$.
\end{prop}

\begin{proof}
Any $B\in \scrR_{A}$ is a block matrix $B=(B_{ij})$ such that
$\mathbb{J}_{i,n_i}(0)B_{ij}+B_{ij}\mathbb{J}_{j,n_j}(0)=0$ for any $1\leq i,j\leq s$. It follows by Lemma~\ref{lem:ALB} that each $B_{ij}$ is an anti-layered block matrix, i.e., of the form \eqref{eq:block al1} or \eqref{eq:block al2}. Introduce the following dense open subset $\calU$ of the vector space $\scrR_{A}$ (and so $\scrR_A =\ov{\calU}$):
\begin{align*}
\calU = &\left\{B=(B_{ij})\in \scrR_{A} \mid B_{ii}\in U_{\mathbb{J}_{i,n_i}(0)}, \right.\\
&\left. \qquad B_{ii}^2 \text{ and }B_{jj}^2 \text{ do not have a common eigenvalues}, \forall 1\leq i\neq j\leq s \right\},
\end{align*}
where $\calU_{\mathbb{J}_{i,n_i}(0)}$ is a dense open subset of $\scrR_{\mathbb{J}_{i,n_i}(0)}$ for each $i$ as given in Proposition~\ref{prop:RJsn}. Note the first condition in $\calU$ implies that all eigenvalues in $B_{ii}$ are distinct for each $i$, and the second condition in $\mathcal U$ implies that $B_{ii}$ and $B_{jj}$ for $j\neq i$  do not share the same or the opposite eigenvalue.
Clearly $\calU$ is a dense open subset of $\scrR_{A}$.

By Lemma \ref{lem:detblock}, $B$ is similar to a block upper-triangular matrix. Following the argument for Proposition~\ref{prop:RJsn} (see \eqref{eq:Ubig}), we see there exists a dense open subset $\calU^{\scrD}$ of $\calU$ such that there exists a unique Jordan block  for $B\in \calU^{\scrD}$ associated to each eigenvalue.

Let $B\in \calU^{\scrD}$, and let $\{b_{ij} \mid 1\le j \le n_i\}$ be the set of eigenvalues of $B_{ii}$, for each $i$. It follows by the above discussion and by Lemma \ref{lem:detblock} that there exists $P\in GL_n$ such that $P^{-1}BP=\diag \big(B'_{11},B'_{22},\ldots,B'_{ss} \big)$ where
\[
B'_{ii}=\diag \big(J_{\lceil \frac{i}{2}\rceil}(b_{i1}),J_{\lfloor \frac{i}{2}\rfloor}(-b_{i1}), \ldots, J_{\lceil \frac{i}{2}\rceil}(b_{in_i}),J_{\lfloor \frac{i}{2}\rfloor}(-b_{in_i}) \big),
\]
where $0\neq b_{ij}\neq \pm b_{kl}$ for all $(i,j)\neq (k,l)$. It follows by Proposition \ref{ACmatrix} that any $C\in \scrR_{P^{-1}BP}$ is of the form
\[
\diag (C_{11},\ldots,C_{1n_1}, C_{21},\ldots,C_{2n_1}\ldots, C_{s1},\ldots,C_{sn_s})
\]
where the $i\times i$ matrix $C_{ij}$ (for $1\leq j\leq n_i, 1\leq i \leq s$) satisfies
\[
C_{ij}\diag \big(J_{\lceil \frac{i}{2}\rceil}(b_{ij}),J_{\lfloor \frac{i}{2}\rfloor}(-b_{ij}) \big)+\diag \big(J_{\lceil \frac{i}{2}\rceil}(b_{ij}),J_{\lfloor \frac{i}{2}\rfloor}(-b_{ij}) \big) C_{ij}=0.
\]
By Lemma~\ref{lem:switch} and Proposition~\ref{prop:block a,-a} we have
$\big(C_{ij},\diag(J_{\lceil \frac{i}{2}\rceil}(b_{ij}),J_{\lfloor \frac{i}{2}\rfloor}(-b_{ij})) \big)
\in \scrZ_{\lfloor \frac{i}{2}\rfloor,0,\lceil\frac{i}{2}\rceil-\lfloor \frac{i}{2}\rfloor}$. Together with Lemma~ \ref{lem:ZZZ12}, this implies that $(C,P^{-1}BP)\in \scrZ_{p,0,n-2p}$.
In particular, we have $(P^{-1}AP,P^{-1}BP) \in \scrZ_{p,0,n-2p}$ and hence $(A,B)\in\scrZ_{p,0,n-2p}$ for any $B\in \calU^{\scrD}$. It follows that $(A,B)\in\scrZ_{p,0,n-2p}$ for any $B\in \scrR_{A} =\ov{\calU^{\scrD}}$.
\end{proof}

\subsection{A covering of $\scrZ_n$}

Now we are ready to formulate the main result of this section. Recall $\TPL_n$ from \eqref{TPL}.

\begin{prop}
  \label{prop:cover the entire space}
We have $\scrZ_n =\bigcup_{(p,m,r)\in \TPL_n} \scrZ_{p,m,r}$.
\end{prop}

\begin{proof}
Let $(A,B)\in \scrZ_n$. Without loss of generality we can assume (after some conjugation) that $A=\diag(A_0, A_1, \ldots, A_s)$, where $A_0$ is nilpotent as in Proposition~\ref{prop:Jordan block 0}, and $A_i$ for $i\ge 1$ consists of Jordan blocks of nonzero eigenvalues $\pm a_i$ (with Jordan blocks of eigenvalues $\pm a_i$ arranged as in Proposition~ \ref{prop:block a,-a}) such that $a_i \neq \pm a_j$ for $i\neq j$.  Then $B=(B_0,B_1,\cdots, B_s)$ of the same block form. By Proposition~\ref{prop:Jordan block 0}, we have $(A_0,B_0) \in \scrZ_{p_0, 0, n_0-2p_0}$, and by Proposition~ \ref{prop:block a,-a}, we have $(A_i,B_i) \in \scrZ_{p_i, n_i-2p_i, 0}$, for $i\ge 1$, for suitable $p_0, n_0$ and $p_i, n_i$.
By applying Lemma \ref{lem:ZZZ12} repeatedly, we obtain $(A,B)\in \scrZ_{p,m,r}$, for $p=\sum_{i=0}^s p_i$, $r=n_0-2p_0,$ and $m=n-2p-r$.
\end{proof}

\section{Equal dimensionality of $\scrZ_{n}//GL_n$}
  \label{sec:GIT}

In this section we shall study the GIT quotient $\scrZ_n//GL_n$, and show that it is of pure dimension $n$. Then we use this to complete the proof that $\scrZ_{p,m,r}$ are the irreducible components of $\scrZ_n$.
\subsection{Generalities of GIT quotients}

The group $G=GL_n$ acts on
\[
\scrD_n=M_n(\bbC)\times M_n(\bbC)
\]
by letting
\[
g\cdot (A,B) =(g Ag^{-1},gBg^{-1}), \qquad
\text{ for }g\in G,(A,B)\in \scrD_n.
\]
Denote by $\bbC[\scrD_n]$ the coordinate ring of $\scrD_n$, and by $\bbC[\scrD_n]^G\subseteq \bbC[\scrD_n]$ the subring of invariants. By a theorem of Hilbert, $\bbC[\scrD_n]^G$ is a finitely generated $\bbC$-algebra since $G$ is reductive. The following result is due to Gurevich, Procesi and Sibirskii; see \cite[Theorem 2.7.9]{Sch07}.

\begin{prop}
  \label{prop:trace}
The invariant ring $\bbC[\scrD_n]^G$ is generated by $\mathrm{Trace}(A^iB^j)$,  for all $i,j\geq0.$
\end{prop}
We consider the GIT quotient of $\scrD_n$ with respect to the action of $G$. A standard reference of geometric invariant theory (GIT) is the book \cite{MFK94}. 
Define
\[
\scrD_n// G:=\mathrm{Specmax}(\bbC[\scrD_n]^G).
\]
The inclusion
$\bbC[\scrD_n]^G\subseteq \bbC[\scrD_n]$ gives rise to a $G$-invariant morphism
\[
\pi: \scrD_n\longrightarrow \scrD_n// G.
\]
Some fundamental results of GIT in our setting are summarized below (see \cite{MFK94}).
\begin{thm}
 \label{fundamental theorem of GIT}
 The following statements hold.
\begin{itemize}
\item[(1)]   If $W_1$ and $W_2$ are two disjoint non-empty $G$-invariant closed subsets of $\scrD_n$, then there is a $G$-invariant function $f\in \bbC[\scrD_n]^G$ such that $f|_{W_1}\equiv1$ and $f|_{W_2}\equiv0$. In particular, the images of $W_1$ and $W_2$ under $\pi$ are disjoint.
\item[(2)] Let $v\in \scrD_n$. Then the orbit closure $\ov{G\cdot v}$ contains a unique closed orbit.
\item[(3)] The map $\pi: \scrD_n\rightarrow \scrD_n// G$ induces a bijection between the set of closed orbits in $\scrD_n$ and the points of $\scrD_n// G$.
\end{itemize}
\end{thm}


%
%
\subsection{The GIT quotient $\scrZ_n//GL_n$}

Since $\scrZ_n\subseteq \scrD_n$ is a $G$-invariant closed subset of a $G$-module, the surjection $\bbC[\scrD_n]\rightarrow \bbC[\scrZ_n]$ is $G$-equivariant and this leads to the following commutative diagrams
\[
\xymatrix{ \bbC[\scrD_n]^G \ar@{->>}[r] \ar@{^(->}[d] & \bbC[\scrZ_n]^G \ar@{^(->}[d] \\
\bbC[\scrD_n] \ar@{->>}[r]  & \bbC[\scrZ_n] }
 \qquad\qquad
 \xymatrix{ \scrZ_n \ar@{^(->}[r] \ar@{->>}[d] & \scrD_n \ar@{->>}[d] \\
\scrZ_n//G \ar@{^(->}[r]  & \scrD_n//G
 }\]


Recall the open subset $\calU_{p,m,r}\subseteq \scrZ_{p,m,r}$ defined in \eqref{eq:UU}.

\begin{lem}
  \label{lem:U polystable}
For $(A,B)\in  \calU_{p,m,r}$, the $G$-orbit of $(A,B)$ in $\scrZ_n$ is closed.
\end{lem}

\begin{proof}
Let $(A,B)\in  \calU_{p,m,r}$. Recall from \eqref{eq:gAg}--\eqref{eq:gBg} that there exists some $g\in G$ such that
\begin{align}
 \label{eq:ABg}
 \begin{split}
gAg^{-1} &={\diag}(a_1,-a_1,\dots, a_p,-a_p,a_{p+1},\dots,a_{p+m},\underbrace{0,\dots,0}_r),
\\
gBg^{-1} &= {\diag}(C_1,\dots,C_p,\underbrace{0,\dots,0}_{m}, b_{p+1},\dots ,b_{p+r}),
 \end{split}
\end{align}
where $a_1,\dots,a_{p+m} \in \bbC^*$ satisfy $a_i\neq \pm a_j$ (for $i\neq j$),
$C_i =\begin{bmatrix} 0& b_i\\ c_i&0
\end{bmatrix}$
with $b_ic_i\neq0$ for $1\leq i\leq p$, and $b_{p+1}, \ldots, b_{p+r}$ are distinct nonzero scalars.

Consider the $\mathbb C$-algebra $R := \mathbb{C} \langle x,y\rangle /(xy+yx)$. Let $\text{mod}_n(R)$ be the variety formed by $n$-dimensional $R$-modules. Then $\text{mod}_n(R)\cong \scrZ_n$. Clearly, the $R$-module corresponding to $(A,B)$, or equivalently to $(gAg^{-1}, gBg^{-1})$, is semisimple by the description \eqref{eq:ABg}, and hence its orbit is closed; see \cite[\S 12.6]{Ar69}.
\end{proof}

For $(p,m,r) \in \TPL_n$, we denote by
\[
\pi_{p,m,r}: \scrZ_{p,m,r}\longrightarrow \scrZ_{p,m,r}//G
\]
the natural map induced by $\bbC[\scrZ_{p,m,r}]^G\subseteq\bbC[\scrZ_{p,m,r}].$ Since the polynomials appearing in the inequalities \eqref{eq1:U}-\eqref{eq4:U} are $G$-invariant, these inequalities define an open subset of $\scrZ_{p,m,r}//G$, which will be denoted by $\widetilde{\calU}_{p,m,r}$. Recall $\calU_{p,m,r}$ is the open subset of $\scrZ_{p,m,r}$ defined in \eqref{eq:UU}.

\begin{thm}
 \label{thm:pure dimension}
The variety $\scrZ_n//G$ is of pure dimension $n$; that is, every irreducible component  $\scrZ_{p,m,r}//G$ has the same dimension $n$, for any $(p,m,r) \in \TPL_n$.
\end{thm}

\begin{proof}
By definition we have $\pi_{p,m,r}(\calU_{p,m,r})=\widetilde{\calU}_{p,m,r}$. Then
\[
\pi_{p,m,r}:\calU_{p,m,r}\longrightarrow \widetilde{\calU}_{p,m,r}
\]
is a surjective regular map between irreducible varieties. Since the $G$-orbit of any point $(A,B)\in \calU_{p,m,r}$ is closed by Lemma~\ref{lem:U polystable}, 
Theorem \ref{fundamental theorem of GIT}(1) shows that the fiber of $\pi_{p,m,r}$ over $\pi_{p,m,r} (A,B) \in \widetilde{\calU}_{p,m,r}$ is isomorphic to the closed $G$-orbit of $(A,B)$ in $\calU_{p,m,r}$.

Let $(A,B)\in \calU_{p,m,r}$.
From the proof of Proposition~ \ref{prop:dimsion of Zprm}, one can read off (up to a suitable conjugation) that the stabilizer subgroup $G_{(A,B)}$ has dimension $n-p$.
So
\[
\dim G\cdot (A,B)=\dim G-\dim G_{(A,B)}=n^2-n+p.
\]
Hence, by Proposition \ref{prop:dimsion of Zprm}, we have
\[
\dim \widetilde{\calU}_{p,m,r}=\dim \calU_{p,m,r}-\dim G\cdot (A,B)=(n^2+p)-(n^2-n+p)=n.
\]
Since $\widetilde{\calU}_{p,m,r}$ is dense open in $\scrZ_{p,m,r}//G$, we have $\dim (\scrZ_{p,m,r}//G)=n$.
\end{proof}

\subsection{Irreducible components of $\scrZ_n$}

\begin{prop}
  \label{prop:distinct}
There is no inclusion relation between any two of the varieties $\scrZ_{p,m,r}$, where $(p,m,r)  \in \TPL_n$.
\end{prop}

\begin{proof}
Since $\scrZ_{p,m,r}$ are irreducible subvarieties of $\scrZ_{n}$ (see Proposition~\ref{prop:irred}), so are $\scrZ_{p,m,r}//G$.

We shall prove it by contradiction.
Suppose we have an inclusion relation between $\scrZ_{p,m,r}$ and $\scrZ_{p',m',r'}$ for some $(p',m',r')\neq (p,m,r)\in \TPL_n$. Then we have an inclusion relation between $\scrZ_{p,m,r}//G$ and $\scrZ_{p',m',r'}//G$. Since $\scrZ_{p,m,r}//G$ and $\scrZ_{p',m',r'}//G$ are irreducible and have the same dimension by Theorem~ \ref{thm:pure dimension}, we must have $\scrZ_{p,m,r}//G =\scrZ_{p',m',r'}//G.$

 As we have either $2p+m\neq 2p'+m'$ or $2p+r\neq 2p'+r'$, we separate in two cases (i)-(ii).

Case (i).
Assume $2p+m\neq 2p'+m'$.
Any $(A,B)$ in the dense open subset $\calU_{p,m,r}\subseteq \scrZ_{p,m,r}$ defined in \eqref{eq:UU} satisfy ${\rm rank}(A)=2p+m$. Similarly, for $(A',B') \in \calU_{p',m',r'}$ of $\scrZ_{p',m',r'}$ , we have ${\rm rank}(A')=2p'+m' \neq \rank (A)$. Hence, we have $\calU_{p,m,r} \cap \calU_{p',m',r'}= \emptyset$. Then it follows that $\widetilde{\calU}_{p,m,r} \cap \widetilde{\calU}_{p',m',r'}= \emptyset,$ which contradicts with the equality $\scrZ_{p,m,r}//G =\scrZ_{p',m',r'}//G$. (Recall $\calU_{p,m,r}$ and $\calU_{p',m',r'}$ are dense open subsets of the irreducible varieties $\scrZ_{p,m,r}//G$ and $\scrZ_{p',m',r'}//G$, respectively.)



Case (ii).
Assume $2p+r\neq 2p'+r'$. The same argument as in Case~(i) by considering the matrix $B$ instead leads to a contradiction.

The proposition is proved.
\end{proof}

Theorem~ \ref{thm:AC} below now follows from Propositions~\ref{prop:irred}, \ref{prop:cover the entire space} and \ref{prop:distinct}.

\begin{thm}
\label{thm:AC}
The variety $\scrZ_n$ has one irreducible component $\scrZ_{p,m,r}$, for each $(p,m,r) \in \TPL_n$, so that $\scrZ_n =\cup_{(p,m,r) \in \TPL_n} \scrZ_{p,m,r}$.
Moreover, we have $\dim \scrZ_{p,m,r} =n^2+p$.
\end{thm}

A simple counting of the cardinality of the finite set $\TPL_n$ gives us the following.

\begin{cor}\label{cor:number of irreducible}
The number of irreducible components in $\scrZ_n$  is $(k+1)^2$ if $n=2k$, and $k(k+1)$ if $n=2k-1$.
\end{cor}

\subsection{Lagrangian}

There exists a natural action of the symmetric group $S_n$ on $(\bbC^2)^n$.
Define a bilinear map $\Omega:(\bbC^2)^n\times (\bbC^2)^n\rightarrow \bbC$ by
\begin{eqnarray*} 
\Omega \big(((x_1,y_1),\ldots,(x_n,y_n)),((x'_1,y'_1),\ldots,(x'_n,y'_n)) \big)
=x_1y'_1-x'_1y_1+\ldots+x_n y'_n -x'_ny_n.
\end{eqnarray*}
Then $\Omega$ is a symplectic bilinear form.

Let $(\bbC^2)^{n}_{gen}$ be the subset of $(\bbC^2)^{n}$ consisting of ordered $n$-tuples of distinct points in $\bbC^2$. Then $(\bbC^2)^{n}_{gen}$ is clearly a $S_n$-invariant subvariety of $(\bbC^2)^{n}$, and the restriction of $\Omega$ to $(\bbC^2)^{n}_{gen}$ is also a symplectic bilinear form. For $(\bbC^2)^{n}_{gen}$, the $S_n$-action is free and preserves the symplectic form too. Hence the quotient $((\bbC^2)^{n}_{gen}/S_n, \Omega)$ is a smooth symplectic variety.

Recall $\calU_{p,m,r}$ is the open set of $\scrZ_{p,m,r}$ defined as \eqref{eq1:U}-\eqref{eq4:U}, and $\pi_{p,m,r}:\calU_{p,m,r}\rightarrow \widetilde{\calU}_{p,m,r}$ the natural projection. We define a morphism
\[
\eta:\widetilde{\calU}_{p,m,r}\longrightarrow (\bbC^2)^{n}_{gen}/S_n
\]
by letting
\begin{align*}
&\eta\Big( (\diag(a_1,-a_1,\dots, a_p,-a_p,a_{p+1},\dots,a_{p+m},\underbrace{0,\dots,0}_r),\\
&\qquad\quad \diag(
\begin{bmatrix} 0& b_1\\ c_1&0 \end{bmatrix},\dots,
\begin{bmatrix} 0& b_p\\ c_p&0 \end{bmatrix},
 \underbrace{0,\dots,0}_m),b_{p+1},\dots,b_{p+r}, \Big)\\
=&\Big((a_1,b_1c_1),(-a_1,b_1c_1),\ldots,(a_p,b_pc_p),(-a_p,b_pc_p),
\\
&  \qquad \qquad (a_{p+1},0),\ldots,(a_{p+m},0),(0,b_{p+1}),\ldots,(0,b_{p+r}) \Big).
\end{align*}
One sees that $\Omega|_{{\rm Im} (\eta) \times {\rm Im} (\eta)}=0$.

\begin{prop}
The morphism $\eta$ is injective (and we identify $\widetilde{\calU}_{p,m,r}$ with the image ${\rm Im} (\eta)$). Then $\widetilde{\calU}_{p,m,r}$ is a lagrangian subvariety of $(\bbC^2)^{n}_{gen}/S_n$.
\end{prop}
\begin{proof}
Let $A=\diag(a,-a)$, $B=\left[\begin{array}{ccccc}
 0 & b  \\
 c & 0
 \end{array}\right]$. Then there exists an inverse matrix $g=\diag(1,c)$ such that $g^{-1}Ag=A$ and $g^{-1}Bg=\left[\begin{array}{ccccc}
 0 & bc  \\
 1 & 0
 \end{array}\right]$ for any $a,b,c\in\bbC^*$.
Therefore $\eta$ is injective. Together with $\Omega|_{{\rm Im} \eta\times {\rm Im} \eta}=0$,
we see that $\eta^*(\Omega)|_{\widetilde{\calU}_{p,m,r}\times \widetilde{\calU}_{p,m,r}}=0$.
Since the dimension of $\widetilde{\calU}_{p,m,r}$ is $n$, $\widetilde{\calU}_{p,m,r}$ is lagrangian.
\end{proof}

\subsection{Commuting vs anti-commuting varieties}
  \label{subsec:C-AC}

Denote by $\scrZ_n^+$ the commuting variety
\[
\scrZ_n^+ =\{ (A, B) \in M_n(\bbC) \times M_n(\bbC) \mid AB=BA \}.
\]
The anti-commuting and commuting varieties can be related directly. We mention two possible such connections.

On one hand, there is a natural morphism $\scrZ_n \rightarrow \scrZ_n^+$, which sends $(A,B) \mapsto (A^2, B)$, or $(A,B) \mapsto (A^2, B^2)$.

On the other hand, the skew-polynomial algebra $\calP_2^-$ is Morita super-equivalent to a polynomial-Clifford algebra $\calP\calC_2$ (cf. \cite{KW09}); that is, there exists a superalgebra isomorphism between the tensor product superalgebra $\calP_2^- \otimes \calC_2$ and $\calP\calC_2$ (Here we recall the Clifford algebra $\calC_2$ is isomorphic to the simple matrix algebra $M_{2\times 2}(\bbC)$). Thus $\scrZ_n$ or the variety of modules of dimension $n$ of the algebra $\calP_2^-$, is isomorphic to  the variety of modules of dimension $2n$ of the algebra $\calP\calC_2$. By restriction the modules of $\calP\calC_2$ to $\calP_2$, we obtain a morphism $\scrZ_n \rightarrow \scrZ_{2n}^+$.

It might be interesting to explore these constructions further.

\section{The semi-nilpotent anti-commuting variety}
  \label{sec:nilpotent}

Recall $G=GL_n$. Introduce the following semi-nilpotent anti-commuting variety, a $G$-subvariety of $\scrZ_n$:
\[
\scrN_n =\{ (A, B) \in M_n(\bbC) \times M_n(\bbC) \mid AB+BA=0, \text{ for } A \text{ nilpotent} \}.
\]
Let $\la =(\la_1, \la_2, \ldots)$  be a partition of $n$. We denote by $\scrN_\la$ the Zariski closure of the union of $G$-orbits of $(J_\la(0),B) \in \scrN_n$.

\begin{lem}\label{lem: irreducible of Nlambda}
For each partition $\la$ of $n$, the variety $\scrN_\la$ is irreducible.
\end{lem}
\begin{proof}
Define the morphism of varieties $\phi_{\la}: \GL_n\times \scrR_{J_\la} \longrightarrow \scrN_{n}$ to be
$$
\phi_{\la}(g, B):=(gJ_\la g^{-1},gBg^{-1}).
$$
Since $\GL_n\times \scrR_{J_\la}$ is an irreducible affine variety and
$\overline{ {\rm Im}(\phi_{\la})} =\scrN_\la$, $\scrN_\la$ is irreducible.
\end{proof}

\begin{lem}\label{lem: dim of Nlambda}
We have $\dim\scrN_\la=n^2$, for each partition $\la$ of $n$.
\end{lem}

\begin{proof}
By the definition, there is a projection map
$$\pi_1: \scrN_\la\longrightarrow \ov{G\cdot J_\la},$$
which is a surjective morphism of varieties.
Introduce the following subvariety of $\scrN_\la$:
\[
U_{\la} =\{(A,B)\in \scrN_\la \mid A \text{ is similar to }J_\la\}.
\]
Since $G\cdot J_\la$ is open in $\ov{G\cdot J_\la}$, we see that $U_{\la}=\pi_1^{-1}( G\cdot J_\la)$  is open in $\scrN_{\la}$, and $\ov{U_\la} =\scrN_\la$  by Lemma \ref{lem: irreducible of Nlambda}.

Write $\la =(\la_1, \la_2, \ldots, \la_\ell)$ of length $\ell$. From Proposition~\ref{ACmatrix} with $\sigma=1$, the fibre of $\pi_1$ over any point in $G\cdot J_\la$ is a linear space of dimension $d_\la :=\la_1+3\la_2+\ldots+(2\ell-1)\la_\ell$,
which coincides with the dimension of the stabilizer $G_{J_\la}$ by Proposition~ \ref{ACmatrix} with $\sigma=1$.
Hence,
\[
\dim U_{\la}=\dim (G\cdot J_\la)+d_\la =\dim (G\cdot J_\la)+\dim G_{J_\la} =\dim G =n^2.
\]
The lemma follows now from $\ov{U_\la} =\scrN_\la$.
\end{proof}

\begin{thm}
 \label{thm:N}
The variety $\scrN_n$ has irreducible components $\scrN_\la$ for all partitions of $n$.
Moreover, $\scrN_\la$ for all $\la$ have the same dimension $n^2$.
\end{thm}

\begin{proof}
It is clear that $\bigcup_{\la} \scrN_\la= \scrN_n.$ Thanks to the irreducibility of $\scrN_\la$ by Lemma \ref{lem: irreducible of Nlambda}, it remains to show that there is no inclusion relation between  $\scrN_\la$ for two different $\la$.
For any two partitions $\la,\mu$ of $n$, if $\scrN_\la\subseteq \scrN_{\mu}$, then $\scrN_\la= \scrN_{\mu}$ by Lemmas~\ref{lem: irreducible of Nlambda} and \ref{lem: dim of Nlambda}. Denote by $\pi_1: \scrN_\la \rightarrow \ov{G\cdot J_\la}$ and $\pi_1': \scrN_{\mu} \rightarrow \ov{G\cdot J_{\mu}}$ the natural projection. Let $U_{\la}=\pi_1^{-1}(G\cdot J_\la)$ and $U_{\mu}=\pi_1'^{-1}(G\cdot J_{\mu})$. Then $U_{\la}$ and $U_{\mu}$ are open in $\scrN_\la$ and $\scrN_{\mu}$ respectively. Since $\scrN_\la= \scrN_{\mu}$, we have $U_{\la}\cap U_{\mu}\neq \emptyset$, and then $\la=\mu$. The theorem is proved.
\end{proof}



\begin{thebibliography}{ABC00}

\bibitem[Ar69]{Ar69}
M. Artin,
\emph{On Azumaya algebras and finite dimensional representations of rings},
J. Algebra {\bf 11} (1969), 532--563.

\bibitem[Ba01]{Ba01}
V. Baranovsky,
{\em The variety of pairs of commuting nilpotent matrices is irreducible},
Transform. Groups {\bf 6} (2001), 3--8.

\bibitem[Boz16]{Boz16}
T. Bozec,
{\em     Quivers with loops and generalized crystals},
Compositio Math. {\bf 152} (2016), 1999--2040. 

\bibitem[Br10]{Br10} M. Brion, {\em Introduction to action of algebraic groups}, Les cours du CIRM {\bf 1}(2010), 1-22.

\bibitem[BPR]{BPR} S. Basu, R. Pollack, and M.-F. Roy,
{\em Algorithms in real algebraic geometry},
Algorithms and Computation in Mathematics, vol. {\bf 10}. Springer-Verlag, Berlin, 2006 (second edition).

\bibitem[CL19]{CL19} X. Chen and M. Lu,
{\em Varieties of modules over $q$-Weyl algebra},
in preparation, 2019.


\bibitem[Ger61]{Ger61} M. Gerstenhaber,
{\em On dominance and varieties of commuting matrices}, Ann. of Math. {\bf 73} (1961), 324--348.

\bibitem[Ha77]{Ha77} R. Hartshorne, Algebraic Geometry, GTM {\bf 52}, Springer-Verlag, 1977.

\bibitem[KW09]{KW09} T. Khongsap and W. Wang,
{\it Hecke-Clifford algebras and spin Hecke algebras IV: odd double affine type},  SIGMA {\bf 5} (2009),
012, 27 pages, \href{https://arxiv.org/abs/0810.2068}{arXiv:0810.2068}.

\bibitem[Li16]{Li16} Y. Li,
{\em Canonical bases of Cartan-Borcherds type, II: consructible functions on singular supports},
Proc. symp. pure math. {\bf 92} (2016), 167--179.

\bibitem[Lu91]{Lu91} G. Lusztig,
{\em Quivers, perverse sheaves, and quantized enveloping algebras},
 J. Amer. Math. Soc. {\bf 4} (1991), 365--421.

\bibitem[MFK94]{MFK94} D. Mumford, J.Fogarty and F.Kirwan,
{\em Geometric Invariant Theory}, Third edition, Ergebnisse der Mathematik
und ihrer Grenzgebiete (2), {\bf 34}, Springer-Verlag, Berlin, 1994.

\bibitem[MT55]{MT55} T.S. Motzkin, O. Taussky,
{\em Pairs of matrices with property L. II}, Trans. Amer. Math.
Soc. {\bf 80} (1955), 387--401.

\bibitem[Pr03]{Pr03}  A. Premet,
{\em Nilpotent commuting varieties of reducible Lie algebras}, Invent. Math. {\bf 154} (2003),  653--683.

\bibitem[Ri79]{Ri79}  R. Richardson,
{\em Commuting varieties of semisimple Lie algebras and algebraic groups}, Compositio Math. {\bf 38} (1979), 311--327.

\bibitem[Sch07]{Sch07} A.H.W. Schmitt,
{\em Geometric Invariant Theory Relative to a Base Curve},
In: Pragacz P. (eds) Algebraic Cycles, Sheaves, Shtukas, and Moduli, 2007,
Trends in Mathematics. Birkh\"{a}user Basel.

\bibitem[WW18]{WW18} J. Wan and W. Wang,
{\em Stability of the centers of group algebras of $GL_n(q)$}, Adv. in Math. (to appear),  \href{https://arxiv.org/abs/1805.08796}{arXiv:1805.08796}.

\end{thebibliography}
\end{document}